\documentclass[11pt,a4paper,reqno]{amsart}
\usepackage[margin=1.0in]{geometry}
\usepackage{amsmath,amsthm,amssymb,amsxtra,comment,graphicx,psfrag}
\usepackage{enumerate,enumitem}
\usepackage{bm,mathrsfs}
\usepackage{mathtools}
\usepackage{xspace}
\usepackage{color}
\usepackage{cite}
\usepackage{subcaption}
\usepackage{multirow}
\usepackage{longtable}
\usepackage{color}
\usepackage{hyperref}
\hypersetup{%
  colorlinks = true,
  citecolor=blue, 
  linkcolor  = black
}

\allowdisplaybreaks

\usepackage{titlesec}
\titleformat{\section}{\vskip10pt\large\bfseries}{\thesection.}{0.5em}{\centering\vspace{5pt}}
\titleformat{\subsection}{\vskip10pt\normalsize\bfseries}{\thesubsection.}{0.5em}{}
\newtheorem{theorem}{Theorem}[section]

\newtheorem{lemma}[theorem]{Lemma}

\newtheorem{remark}[theorem]{Remark}
\theoremstyle{definition}

\def\bfx{{\bf x}}
\def\bfe{{\bf e}}

\def\nu{n}

\def\vb{{\bf v}}

\def\d{{\mathrm d}}

\def\R{{\mathbb R}}

\def\ud{\underline{D}}
\def\md{\partial_{t}^\bullet}

\def\ehm{{\hat{e}_h^m}}
\def\eM{{e_h^{m+1}}}
\def\ehM{{\hat{e}_h^{m+1}}}

\def\Gm{{\Gamma^m}}

\def\Ghm{{\Gamma^m_h}}
\def\GhM{{\Gamma^{m+1}_h}}
\def\Ghsm{{\hat\Gamma^m_{h, *}}}
\def\GhsM{{\hat\Gamma^{m+1}_{h, *}}}

\def\Ghso{{\Gamma_{h,\rm f}^0}}

\def\Hm{{H^m}}

\def\nm{{n^m}}

\def\nhm{{n^m_h}}

\def\nhsm{{\hat n^m_{h, *}}}

\def\nbhsm{{\bar n^m_{h, *}}}

\def\nsm{{n^m_{*}}}
\def\nsM{{n^{m+1}_{*}}}

\def\Thsm{{\hat T^m_{h, *}}}

\def\Tbhsm{{\bar T^m_{h, *}}}

\def\Tsm{{T^m_{*}}}

\def\Nhsm{{\hat N^m_{h, *}}}

\def\Nsm{{N^m_{*}}}
\def\NsM{{N^{m+1}_{*}}}

\def\Nbhsm{{\bar N^m_{h, *}}}

\def\Tbhm{{\bar{T}^m_h}}

\def\nbhm{{\bar{n}^m_h}}

\numberwithin{equation}{section}

\begin{document}
\title[]{\parbox[b]{\linewidth+20pt}{\centering Stabilization and optimal $L^2$ convergence of Dziuk's method with piecewise linear parametric finite elements for curve-shortening flow}}

\author[]{Genming Bai}
\address{Genming Bai: Department of Mathematics and Statistics, Old Dominion University,
Norfolk, VA, USA. {\it Email address: \tt gbai@odu.edu}}



\allowdisplaybreaks

\maketitle
\vspace{-20pt}

\begin{abstract}\noindent
	{\small 
		We propose a stabilized version of the fully discrete Dziuk's method for the curve-shortening flow of a closed planar curve with piecewise linear parametric finite elements.
		With a carefully designed stabilization term, we are able to show a surprising discrete tangential stability of the Barrett--Garcke--N\"urnberg (BGN) type under the parabolic scaling $\tau\simeq h^2$---a feature hidden at the continuous level.
		Together with a new super-approximation result for the reversely averaged normal vector of linear elements, this Dziuk-type discrete tangential stability yields optimal $L^2$ convergence.
	}
\end{abstract} 

\setlength\abovedisplayskip{3.5pt}
\setlength\belowdisplayskip{3.5pt}

\section{Introduction}\label{sec:intro}

Let $\Gamma(t)\subset\mathbb R^2$, $t\in[0,T]$, be a smooth family of closed
curves evolving by curve-shortening flow.  With a fixed choice of unit normal
$n$ and mean curvature $H$, we consider the evolution system
\begin{align}\label{eq:PDE}
	\partial_t X=-Hn=\Delta_{\Gamma(t)} {\rm id}_{\Gamma(t)},
\end{align}
where $X(\cdot,t):\Gamma^0\to\Gamma(t)$ is the parametrization of the
flow.

Dziuk's parametric finite element method, introduced in the seminal work
\cite{Dziuk1991}, is one of the fundamental numerical approaches to
curvature-driven evolution. It discretizes the weak formulation of the evolution system \eqref{eq:PDE}. At the continuous level, however, \eqref{eq:PDE} lacks control in the tangential direction and hence is not strictly parabolic; see \cite[Eq.~(2.1)]{CY2007}. The same degeneracy appears at the numerical level as well. As studied in \cite{BL2023,Li2021,Li2020}, the trajectory error of \eqref{eq:PDE} exhibits parabolicity only in the normal direction and provides no tangential control.
This partial parabolicity is the main obstacle in the way of analyzing the discretization of \eqref{eq:PDE}.

A projection-error framework was subsequently developed in
\cite{BL2024,BL2025,BGV2026,BGV2027} within the standard parametric finite element framework of
\cite{DDE2005,DE2013,KLL19,BGN2020}. Instead of comparing prescribed
particle trajectories, this approach measures the distance between the numerical surface and its closest-point projection onto the exact surface. In terms of this projection error, full parabolicity can be recovered. Although unified, the main difficulty is that the inverse inequalities used to control the nonlinear geometric terms generally require finite elements of sufficiently high degree. For example, the convergence analyses in \cite{BL2024,Li2020} require polynomial degree
$k\geq3$. The purpose of the present paper is to extend the projection-error framework to the lowest-order case $k=1$ and to derive an optimal $L^2$ error estimate for the first time in this regime.

Aside from the unified parametric finite element analysis discussed above, ad hoc analysis was developed for $k=1$ as well.
In \cite{Dziuk1994}, Dziuk introduced a length-element analysis through an equivalent finite-difference formulation and proved a suboptimal $L^2$ error estimate for a semidiscrete scheme. Pozzi and Stinner \cite{PS2017} subsequently refined this approach and established super-convergence estimates leading to optimal $H^1$ convergence. Using techniques from \cite{Dziuk1994,PS2017,Li2020}, Ye and Cui \cite{YC2021} extended the analysis to a fully discrete scheme and obtained an optimal $H^1$ error estimate. The length-element analysis of \cite{Dziuk1994,PS2017} has also been adapted to modified curve-shortening models in \cite{JSZ2023}. For a tangentially regularized version of Dziuk's method, optimal $H^1$ convergence for piecewise linear elements was proved in \cite{BDS2017,EF2017,DD1994} using super-convergence estimates. In the regularized setting, the system is fully parabolic, so an analysis based on the length element is not required.

In this paper, we overcome the lowest-order obstruction within the
parametric finite element framework
\cite{DDE2005,DE2013,KLL19,BGN2020} and the projection-error framework \cite{BL2024,BL2025,BGV2026,BGV2027}. The two principal ingredients are a
reversely weighted normal and the introduction of a consistent stabilization term.
Let $t_m=m\tau$ with $\tau$ being the uniform timestep size, let $\Gamma_h^m$ be the current polygonal curve, and
denote by $S_h(\Gamma_h^m)$ the space of continuous piecewise linear
finite element functions on $\Gamma_h^m$. We write $I_h$ for the nodal
interpolation operator and use a superscript $h$ on an integral to
denote mass lumping. At a node $p$ shared by two adjacent elements, let
$n_h^m(p-)$ and $n_h^m(p+)$ denote the two one-sided unit normals, and
let $\ell_{h,-}^m$ and $\ell_{h,+}^m$ denote the lengths of the
corresponding elements. We define the reversely weighted normal
by
\begin{equation}\label{eq:intro-reverse-normal}
	\bar n_h^m(p)
	=
	\frac{
		\ell_{h,+}^m n_h^m(p-)
		+
		\ell_{h,-}^m n_h^m(p+)
	}{
		\ell_{h,-}^m+\ell_{h,+}^m
	}.
\end{equation}
Thus, each one-sided normal is weighted by the length of the opposite
element.

We propose the following stabilized Dziuk's method:
Given $\Gamma_h^m$, the new parametrization
$X_h^{m+1}\in[S_h(\Gamma_h^m)]^2$ is determined by
\begin{align}\label{eq:Dziuk-stab}
	&\int_{\Gamma_h^m}^{h}
	\frac{X_h^{m+1}-{\rm id}}{\tau}\cdot\phi_h
	+
	\int_{\Gamma_h^m}
	\nabla_{\Gamma_h^m}X_h^{m+1}\cdot
	\nabla_{\Gamma_h^m}\phi_h
	\notag\\
	&\qquad=
	\int_{\Gamma_h^m}
	\nabla_{\Gamma_h^m}{\rm id}\cdot
	\nabla_{\Gamma_h^m}
	I_h\!\left[\phi_h-(\phi_h\cdot\bar n_h^m)\bar n_h^m\right]
	\qquad
	\forall\,\phi_h\in[S_h(\Gamma_h^m)]^2,
\end{align}
and $\Gamma_h^{m+1}:=X_h^{m+1}(\Gamma_h^m)$.  If the right-hand side is
omitted, \eqref{eq:Dziuk-stab} reduces to the fully discrete counterpart of Dziuk’s semidiscrete scheme considered in \cite{Dziuk1994}.  The additional stabilization term is consistent because, on a smooth closed
curve,
\[
\int_\Gamma\nabla_\Gamma{\rm id}\cdot
\nabla_\Gamma\!\left[(I-n\otimes n)\phi\right]
=
\int_\Gamma Hn\cdot(I-n\otimes n)\phi
=0.
\]
A stabilization term of this form was first introduced for BGN-type methods in \cite{BL2025,BGV2026}. Surprisingly, it also enhances the
stability of Dziuk's method when and only when $k=1$ and the parabolic scaling $\tau\simeq h^2$ is imposed.



The main contributions of this paper are summarized as follows.
\begin{enumerate}[label=\textup{(\roman*)}]
	\item
	We construct and analyze the stabilized piecewise linear method
	\eqref{eq:Dziuk-stab}. The reversely weighted normal $\nbhsm$
	on the consistency curve, defined in \eqref{bar-n-hat-n} as the
	counterpart of \eqref{eq:intro-reverse-normal}, satisfies the
	second-order nodal super-approximation property
	\[
	\nbhsm(p)=n(p)+O(h^2),
	\]
	at any finite element node $p$ on the consistency curve $\Ghsm$ which is defined in Section \ref{subsec:numerical-and-interpolated-curves}. 
	This additional approximation order plays a key role in closing the stability loop. In particular, it helps establish the discrete tangential stability \eqref{eq:tan_stab_H1_e} for the lowest-order finite element.
	
	
	\item
	We derive a discrete tangential smoothing mechanism for the lowest-order case. 
	At the continuous level, the evolution system \eqref{eq:PDE} implies
	\begin{align}\label{eq:tan-stab-cont}
		(I-n\otimes n) \partial_t X = 0.
	\end{align}
	Unlike the BGN method, the discrete analogue of the vanishing right-hand side of \eqref{eq:tan-stab-cont} is not clear:
	\begin{align}\label{eq:tan-stab-disc}
		(I-n\otimes n) \partial_t e =\, ?.
	\end{align}
	Specifically, the continuous relation \eqref{eq:tan-stab-cont} provides no information regarding what should replace the placeholder $``?"$ in \eqref{eq:tan-stab-disc}.
	In other words, discrete tangential stability cannot be directly inferred from the vanishing right-hand side at the continuous level.
	Whether this placeholder $``?"$ remains numerically stable is the core question addressed in this paper.
	We provide a positive answer in Section \ref{sec:tan_stab}.
	This Dziuk-type tangential stability arises from a delicate interplay among the super-approximation of $\nbhsm$, mass lumping, and the parabolic scaling $\tau\simeq h^2$.
	Notably, for higher scaling regimes where $\tau\simeq h^{k+1}$ with $k\geq 2$, this discrete tangential stability is lost. 
	In such regimes, the mass bilinear form overwhelms the stiffness bilinear form, the latter of which serves as the fundamental source of the discrete tangential stability; see Section \ref{sec:tan_stab} for more details.
	
	It should also be noted that discrete tangential stability is indispensable for the convergence proof. If one relies solely on the standard velocity estimates presented in Section \ref{sec:bbd_vel}, a power of the consistency error is lost, preventing the closure of the discrete stability loop.
	
	
	\item
	We establish the optimal $L^2$ convergence rate of $O(\tau+h^2)$ for the scheme \eqref{eq:Dziuk-stab} with linear finite elements under the parabolic scaling $\tau\simeq h^2$.
	This improves upon the existing $H^1$-optimal convergence rate in  \cite{Dziuk1994,PS2017,YC2021}, which rely on finite-difference analysis. 
	More importantly, our analysis is developed within the unified projection-error framework of \cite{BL2024,BL2025,BGV2026,BGV2027}. It avoids finite-difference characterization and is therefore not tailored specifically to piecewise linear elements.
	
	
\end{enumerate}


The remainder of the paper is organized as follows.
Section~\ref{section:preparation} introduces the parametric finite element
setting, the consistency and closest-point-projected curves, the
reverse-weighted normals, the geometric identities, and the main theorem.
Section~\ref{sec:cons_err} establishes the consistency estimates.
Section~\ref{sec:stab} develops the discrete $H^1$ parabolicity
decomposition, the tangential and full velocity estimates, the displacement
and norm-conversion bounds, the main energy estimate, and the uniform
shape-regularity argument.  The appendices collect the notation,
surface-calculus identities, super-approximation estimates, discrete norm
equivalences, and the proof of the norm-conversion lemma.

\section{Notation, geometric framework and the main result}
\label{section:preparation}

%
%
%
%

\subsection{Parametric finite element method}
\label{subsec:parametric-finite-elements}


Let $\hat K=[0,1]\subset\R$ be the reference interval and let
$\mathbb P_1(\hat K)$ denote the space of polynomials of degree
$1$ on $\hat K$. An order-$1$ parametric triangulation $\mathcal T$ of a globally continuous and piecewise smooth
closed curve with domain $D(\mathcal T)\subset\mathbb R^2$ is a collection of elements
$K$ of the form
\[
K=\hat F_K(\hat K),
\qquad
\hat F_K\in[\mathbb P_1(\hat K)]^2.
\]
Since the curve is globally continuous,
the parametrization is assumed to be compatible at common endpoints.
We write
$h:=\max_{K\in\mathcal T}\operatorname{diam}K$ for the mesh size and denote by $\mathcal N(\mathcal T)\subset\R^2$ the set of nodes of the discrete curve $D(\mathcal T)$. 
The associated parametric finite element space is
defined by
\begin{equation}
	S_h(\mathcal T)
	:=
	\bigl\{
	v_h\in H^1(D(\mathcal T))\cap C^0(D(\mathcal T)):
	v_h|_K\circ \hat F_K\in\mathbb P_1(\hat K)
	\quad\forall K\in\mathcal T
	\bigr\}.
	\notag
\end{equation}
$\R^q$-valued spaces are denoted by $[S_h(\mathcal T)]^q$. We also use
$W_h^{s,p}(D(\mathcal T))$ for the corresponding broken Sobolev space, whose
norm is obtained from the elementwise $W^{s,p}$ norms.

Two order-$1$ parametric triangulations $\mathcal T_1$ and $\mathcal T_2$ are called equivalent, denoted by $\mathcal T_1 \cong \mathcal T_2$, if they have the same number of elements and there is a one-to-one correspondence between $\mathcal T_1$ and $\mathcal T_2$, i.e., for each $K_1\in\mathcal T_1$ there exists a unique $K_2\in\mathcal T_2$ so that the glue of the piecewise-defined transition map
\[
f_h|_{K_1}
:=
\hat F_{2,K}\circ \hat F_{1,K}^{-1}
:
K_1\rightarrow K_2
\]
defines a $W^{1,\infty}$ homeomorphism
\[
f_h:D(\mathcal T_1)\rightarrow D(\mathcal T_2).
\]
Equivalently, $f_h\in[S_h(\mathcal T_1)]^2$ and $\mathcal T_2$ is the
parametric image of $\mathcal T_1$ under $f_h$. Such a map induces the
canonical identification
\begin{equation}
	v_h\in S_h(\mathcal T_1)
	\quad\longleftrightarrow\quad
	v_h\circ f_h^{-1}\in S_h(\mathcal T_2).
	\label{eq:nodal-identification}
\end{equation}
This identification preserves the nodal vector. Conversely, a nodal vector
has a unique realization as a finite element function on each member of a
family of equivalent triangulations. Throughout the paper, finite element
functions on equivalent curves are identified in this way. 
For example, the two integrands of $$\int_{D(\mathcal T_1)} \phi_h \quad\mbox{and}\quad \int_{D(\mathcal T_2)} \phi_h$$ have the same vector of nodal values, denoted by $\vb$, but are defined on different domains $D(\mathcal T_1)$ and $D(\mathcal T_2)$. When the underlying domain is specified, $\vb$ is automatically realized to a finite element function $\phi_h$ on that domain. Since all of the quantitative computations in this paper involve either integrals or norms, our notation for finite element functions will always have a unique and clear meaning. For another example, $\| \phi_h \|_{L^2(D(\mathcal T_1))}$ and $\| \phi_h \|_{L^2(D(\mathcal T_2))}$ denote the norms of finite element functions with the same nodal vector on the two different domains $D(\mathcal T_1)$ and $D(\mathcal T_2)$, respectively.

\subsection{Lagrange interpolation and mass lumping}
\label{subsec:Lagrange-interpolation}

Let
\[
I_{\hat K}:C^0(\hat K)\rightarrow\mathbb P_1(\hat K)
\]
be the degree-$1$ Lagrange interpolation operator on the flat reference domain $\hat K$. For an order-$1$ parametric triangulation
$\mathcal T$, the global interpolation operator
\[
I_h(\mathcal T):
C^0(D(\mathcal T))
\rightarrow
S_h(\mathcal T)
\]
is defined elementwise by
\begin{equation}
	\bigl(I_h(\mathcal T)v\bigr)|_K
	:=
	I_{\hat K}[v\circ \hat F_K]\circ \hat F_K^{-1},
	\qquad K\in\mathcal T.
	\notag
\end{equation}
The definition is applied componentwise to vector- and matrix-valued
functions.

The interpolation operator is compatible with the identification
\eqref{eq:nodal-identification}. More precisely, if $\mathcal T_2$ is
parametrized over $\mathcal T_1$ by $f_h$, then
\begin{equation}
	I_h(\mathcal T_2)(v\circ f_h^{-1})
	=
	\bigl(I_h(\mathcal T_1)v\bigr)\circ f_h^{-1}.
	\notag
\end{equation}
Consequently, the interpolants on two equivalent triangulations have the
same nodal vector. We shall therefore omit the triangulation argument and
write simply $I_h$ whenever its domain is unambiguous.


With the Lagrange interpolation operator above, the corresponding mass-lumped integral is defined as
\begin{equation}
	\int_{D(\mathcal T)}^h f
	:=
	\int_{D(\mathcal T)} I_h f
	=
	\sum_{K\in\mathcal T}\frac{|K|}{2}
	\sum_{p\in\mathcal N(K)}
	f(p) .
	\label{eq:mass-lumped-inner-product}
\end{equation}

\subsection{Shape regularity constants and the interpolation error estimates}

The approximation properties of $I_h(\mathcal T)$ are closely related to the shape regularity constants.
Suppose that we have a parametric triangulation $\mathcal T$ parametrized by $F(\mathcal T)\in S_h(\mathcal{T}_0)^2$ over some canonical reference triangulation $\mathcal{T}_0$. Typically, $\mathcal{T}_0$ is chosen as the flat triangulation $\Gamma_{h,{\rm f}}^0$, defined later in \eqref{eq:flat-reference-curve}. 
The associated shape regularity quantities are defined as
\begin{align}\label{P}
	\begin{aligned}
		\kappa_*(\mathcal T) 
		:= \| F(\mathcal T) \|_{W^{1,\infty}(D(\mathcal T_0))} 
		+ \| F(\mathcal T)^{-1}\|_{W^{1,\infty}(D(\mathcal T))} 
		.
	\end{aligned}
\end{align}
The following standard Lagrange interpolation error estimates hold due to the local approximation property on the reference triangulation $\mathcal T_0$ and the chain rule for differentiation.
\begin{lemma}\label{lemma:Ih}
	Given a parametric triangulation $\mathcal T$ and its underlying domain $D(\mathcal T)$,
	for any function $f\in W_h^{2,p}(D(\mathcal T)) \cap C^0(D(\mathcal T))$ and $p\in[1,\infty]$, we have
	\begin{align}
		&\| (1-I_h) f \|_{L^{p}(D(\mathcal T))}
		+
		h
		\| \nabla_{D(\mathcal T)} (1-I_h) f \|_{L^{p}(D(\mathcal T))}
		\leq
		C_{\kappa_*(\mathcal T)}
		h^{2}\| f \|_{W_h^{2,p}(D(\mathcal T))} .
		\notag
	\end{align}
\end{lemma}
Pullback by $F$ and $F^{-1}$ gives the following norm equivalence
\[
C_{\kappa_*}^{-1}\|v_h\|_{W^{1,p}(D(\mathcal T_0))}
\leq
\|v_h\|_{W^{1,p}(D(\mathcal T))}
\leq
C_{\kappa_*}\|v_h\|_{W^{1,p}(D(\mathcal T_0))}.
\]

\subsection{Norm equivalence on different triangulations}
\label{subsec:norm-equivalence}

Given two equivalent triangulations $\mathcal T_1 \cong \mathcal T_2$ with the transition map $f_h:D(\mathcal T_1)\rightarrow D(\mathcal T_2)$, we can define the linearly interpolated triangulation $\mathcal T_{1,2}^\theta := (1-\theta) \mathcal T_1 + \theta \mathcal T_2$, parametrized by $f_h^\theta:=(1-\theta){\rm id}_{D(\mathcal T_1)}+\theta f_h:D(\mathcal T_1)\rightarrow D(\mathcal T_{1,2}^\theta)$.
The canonical nodal identification allows a finite element function to be
realized on any triangulation $\mathcal T_{1,2}^\theta$. The following result,
proved in \cite[Lemma 4.3]{KLL17} and \cite[Lemma 7.2]{KLL19}, shows that the corresponding norms are
uniformly equivalent when the transition map is sufficiently small.

\begin{lemma}\label{lemma:norm-equiv}
	Suppose that
	\[
	\|\nabla_{D(\mathcal T_1)}(f_h-{\rm id}_{D(\mathcal T_1)})\|_{L^\infty(D(\mathcal T_1))}
	\leq \frac12.
	\]
	Then, uniformly for $\theta\in[0,1]$ and for $1\leq p\leq\infty$, every
	finite element function $v_h$ with a common nodal vector satisfies
	\begin{align*}
		c^{-1}\|v_h\|_{L^p(D(\mathcal T_1))}
		&\leq
		\|v_h\|_{L^p(D(\mathcal T_{1,2}^\theta))}
		\leq
		c
		\|v_h\|_{L^p(D(\mathcal T_1))},
		\\
		c^{-1}
		\|\nabla_{D(\mathcal T_1)}v_h\|_{L^p(D(\mathcal T_1))}
		&\leq
		\|\nabla_{D(\mathcal T_{1,2}^\theta)}v_h
		\|_{L^p(D(\mathcal T_{1,2}^\theta))}
		\leq
		c
		\|\nabla_{D(\mathcal T_1)}v_h\|_{L^p(D(\mathcal T_1))},
	\end{align*}
	where $c$ is a universal constant which is independent of $\theta$, $p$, $\mathcal T_1$ and $\mathcal T_2$.
\end{lemma}

\subsection{The numerical curve $\Ghm$ and the consistency curve $\Ghsm$}
\label{subsec:numerical-and-interpolated-curves}


Let $\Gamma_h^0$ be the initial
parametric curve with the triangulation $\mathcal T_{h}^0$. For each element $K^0\subset\Gamma_h^0$, let
$K_{\rm f}^0$ be the flat segment with the same endpoints as $K^0$.
These flat segments form the
piecewise linear reference curve
\begin{equation}
	\Ghso
	:=
	\bigcup_{K^0\subset\Gamma_h^0}K_{\rm f}^0,
	\label{eq:flat-reference-curve}
\end{equation}
with the underlying triangulation $\mathcal T_{h, {\rm f}}^0$.
The connectivity and node numbering on $\Ghso$ are inherited from
$\Gamma_h^0$. In the present piecewise linear case, $k=1$, the curves
$\Ghso$ and $\Gamma_h^0$ coincide. Retaining the notation $\Ghso$ makes the
parametric identifications at later time levels explicit.

Let
\[
{\bf x}^m=(x_1^m,\ldots,x_J^m)^\top
\]
be the nodal vector produced by the fully discrete method \eqref{eq:Dziuk-stab} at time
$t_m=m\tau$, $m=0,...,\lfloor T/\tau \rfloor$. We denote by
\[
X_h^m:\Ghso\longrightarrow\mathbb R^2
\]
the unique function in $[S_h(\Ghso)]^2$ with nodal vector ${\bf x}^m$ and
define
\begin{equation}
	\mathcal T_h^m
	:=
	\bigl\{
	X_h^m(K_{\rm f}^0):
	K_{\rm f}^0\subset\Ghso
	\bigr\},
	\qquad
	\Ghm
	:=
	D(\mathcal T_h^m)
	=
	X_h^m(\Ghso).
	\label{eq:numerical-parametric-curve}
\end{equation}
Whenever $X_h^m$ is regular, injective elementwise and has no self-intersections, $\mathcal T_h^m$ is a
parametric triangulation equivalent to the reference triangulation $\mathcal T_{h, {\rm f}}^0$. Under
the nodal identification introduced in \eqref{eq:nodal-identification}, the
realization of $X_h^m$ on $\Ghm$ is the identity map ${\rm id}_{\Gamma_h^m}$. Likewise, the function
$X_h^{m+1}$ in \eqref{eq:Dziuk-stab}, when realized on $\Ghm$, is the local
discrete flow map from $\Ghm$ to $\Gamma_h^{m+1}$.

Let $\Gm=\Gamma(t_m)$ be the exact curve and let
\[
a^m:
U_\delta(\Gm)
\longrightarrow
\Gm
\]
be the distance retraction, defined on a tubular neighborhood
$U_\delta(\Gm)$. Whenever $\Ghm\subset U_\delta(\Gm)$, set
\begin{equation}
	\hat{\bf x}_*^m
	:=
	\bigl(
	a^m(x_1^m),\ldots,a^m(x_J^m)
	\bigr)^\top.
	\label{eq:projected-nodal-vector}
\end{equation}
The consistency parametric map
$\hat X_{h,*}^m\in[S_h(\Ghso)]^2$ is the unique finite element function with
nodal vector $\hat{\bf x}_*^m$; equivalently,
\begin{equation}
	\hat X_{h,*}^m
	=
	I_h\bigl(a^m\circ X_h^m\bigr)
	\qquad\hbox{on }\Ghso.
	\label{eq:interpolated-parametric-map}
\end{equation}
We then define
\begin{equation}
	\hat{\mathcal T}_{h,*}^m
	:=
	\bigl\{
	\hat X_{h,*}^m(K_{\rm f}^0):
	K_{\rm f}^0\subset\Ghso
	\bigr\},
	\qquad
	\Ghsm
	:=
	D(\hat{\mathcal T}_{h,*}^m)
	=
	\hat X_{h,*}^m(\Ghso).
	\label{eq:interpolated-parametric-curve}
\end{equation}

\subsection{Lift and the inverse lift}
\label{subsec:lift-and-inverse-lift}

Suppose that $\Ghsm$ is sufficiently close to $\Gm$, then
\[
a^m|_{\Ghsm}
:
\Ghsm
\longrightarrow
\Gm
\]
is a bijection. If $v$ is a function on $\Ghsm$, its lift to $\Gm$ is
defined by
\begin{equation}
	v^\ell
	:=
	v\circ(a^m|_{\Ghsm})^{-1}.
	\notag
\end{equation}
Conversely, if $f$ is a function on $\Gm$, its inverse lift to $\Ghsm$ is
\begin{equation}
	f^{-\ell}
	:=
	f\circ a^m.
	\notag
\end{equation}

The fact that every node of $\Ghsm$ lies on $\Gm$ implies a useful
interpolation identity. Indeed, for
\[
K=F_K(K_{\rm f}^0)\subset\Ghsm,
\]
where $F_K$ is the local parametrization map of $K$ on $K_{\rm f}^0$, we have $a^m\circ F_K=F_K$ at all nodes of $K_{\rm f}^0$. Hence
\begin{equation}
	I_{K_{\rm f}^0}[a^m\circ F_K]
	=
	F_K,
	\qquad
	I_h a^m
	:=
	I_{K_{\rm f}^0}[a^m\circ F_K]\circ F_K^{-1}
	=
	{\rm id}
	\quad\hbox{on }K.
	\label{eq:interpolation-of-distance-projection}
\end{equation}

\subsection{Approximation properties of $\Ghsm$}
\label{section:interpolated}

The constants in the interpolation and geometric perturbation estimates
depend on the regularity of the consistency triangulations. In view of \eqref{P}, for
$0\leq l\leq M=\lfloor T/\tau \rfloor$, we quantify this dependence by
\begin{align}\label{PP}
	\begin{split}
		\kappa_l
		:=&\max_{0\le m\le l}
		\big( \| \hat X_{h,*}^m\|_{W^{1,\infty}(\Gamma_{h,\rm f}^0)}
		+ \| (\hat X_{h,*}^m)^{-1} \|_{W^{1,\infty}(\Ghsm)} \big) 
		\\
		=&\max_{0\le m\le l}
		\big( \max_{K\subset\hat\Gamma_{h,*}^m} \|F_K\|_{W^{1,\infty}(K_{\rm f}^0)}
		+ \max_{K\subset\hat\Gamma_{h,*}^m} \|F_{K}^{-1}\|_{W^{1,\infty}(K)} \big) 
		.
	\end{split}
\end{align}

Since the nodes of $\Ghsm$ lie on $\Gamma^m$, the map
$\hat X_{h,*}^m$ is the Lagrange interpolant of
$a^m\circ\hat X_{h,*}^m$ on $\Gamma_{h,\rm f}^0$. Standard interpolation estimates (Lemma~\ref{lemma:Ih}) therefore imply
\begin{align}
	\|a^m\circ \hat X_{h,*}^m - \hat X_{h,*}^m
	\|_{L^\infty(\Gamma_{h,\rm f}^0)}
	+ h \|a^m\circ \hat X_{h,*}^m - \hat X_{h,*}^m
	\|_{W^{1,\infty}(\Gamma_{h,\rm f}^0)}
	&\le C_{\kappa_l}h^{2}.
	\notag
\end{align}

Let $n^m$ and $H^m$ denote the unit normal and the curvature of $\Gamma^m$,
respectively. Their normal extensions to $U_\delta(\Gamma^m)$ are
\[
n_*^m:=n^m\circ a^m,
\qquad
H_*^m:=H^m\circ a^m.
\]
Denote by $\hat n_{h,*}^m$ the unit normal vector of $\Ghsm$.
Since the approximation of normal vector is in general one order worse than the position approximation, we have
\begin{align}\label{normal-intpl}
	\|\hat n_{h,*}^m - n_*^m\|_{L^\infty(\Ghsm)}
	&\le C_{\kappa_l}h.
\end{align}


\subsection{Reversely weighted normals}
\label{subsec:averaged-normal-basic}

We next define the reversely weighted normals on both the
consistency and numerical curves. Let $p\in\mathcal N(\Ghsm)$ be shared by
the adjacent elements $K_-$ and $K_+$, with the traces from $K_-$ and $K_+$
denoted by $p-$ and $p+$, respectively. Set
\[
\ell_{h,*,i}^m:=|w_{K_i}(p)|\,|K_{i\rm f}^0|,
\qquad
w_K:=\nabla_{K_{\rm f}^0}F_K\circ F_K^{-1}.
\]
The reversely weighted normal $\bar n_{h,*}^m\in[S_h(\Ghsm)]^2$ is the
unique finite element function satisfying
\begin{align}
	\bar n_{h,*}^m(p)
	&=
	\frac{\ell_{h,*,+}^m\,\hat n_{h,*}^m(p-)
		+\ell_{h,*,-}^m\,\hat n_{h,*}^m(p+)}{\ell_{h,*,-}^m+\ell_{h,*,+}^m}.
	\label{bar-n-hat-n}
\end{align}
Thus, the normal on either element is weighted by the length of the opposite
element. Similarly, we can define $\ell_{h,i}^m$, $i=\pm$, on $\Ghm$ whose unit normal vector is $\nhm$.
Then, we define
\begin{align}
	\bar n_h^m(p)
	&=
	\frac{\ell_{h,+}^m n_h^m(p-)
		+\ell_{h,-}^m n_h^m(p+)}{\ell_{h,-}^m+\ell_{h,+}^m}.
	\label{eq:bar_n}
\end{align}
This is the averaged normal used in the scheme
\eqref{eq:Dziuk-stab}.

Since the two one-sided normals have unit length, elementary algebra gives (cf. \cite[Eqs. (3.18) and (4.33)]{BL2024})
\begin{subequations}
\begin{align}
	| \bar n_{h,*}^m(p) | &\leq 1,
	&
	\big| | \bar n_{h,*}^m(p) | - 1\big|
	&\le C | \hat n_{h,*}^m(p+) - \hat n_{h,*}^m(p-) |^2
	\le C_{\kappa_l}h^2,
	\label{eq:nhs_bar_len0}\\
	| \bar n_h^m(p) | &\leq 1,
	&
	\big| | \bar n_h^m(p) | - 1\big|
	&\le C | n_h^m(p+) - n_h^m(p-) |^2. \label{eq:nhs_bar_len}
\end{align}
\end{subequations}

The unit normals on $\Ghm$ and $\Ghsm$, i.e., $\nhm$ and $\nhsm$, can be compared along the
intermediate curves introduced above in Section \ref{subsec:norm-equivalence}. Item~7 of Lemma~\ref{lemma:ud}, the
fundamental theorem of calculus, Lemma~\ref{lemma:lump}, and norm equivalence (Lemma~\ref{lemma:norm-equiv})
give
\begin{align}\label{normal-intpl-2}
	\|n_{h}^m - \hat n_{h,*}^m\|_{L^p(\Ghsm)}
	+ \|n_{h}^m - \hat n_{h,*}^m\|_{L_h^p(\Ghsm)}
	&\lesssim \|\nabla_\Ghsm\hat e_h^m\|_{L^p(\Ghsm)}
	\quad\forall p\in[1,\infty]
	.
\end{align}
Here the piecewise polynomial function $n_h^m$ on $\Ghm$ is identified as a piecewise polynomial function on
$\Ghsm$ according to our convention. Applying the same argument to the nodal formulas
\eqref{bar-n-hat-n} and \eqref{eq:bar_n}, including the variation of their
length weights, yields (cf. \cite[Eq. (4.31)]{BL2025})
\begin{align}\label{normal-intpl-3}
	\| \nbhm -\nbhsm \|_{L^p(\Ghsm)}
	\lesssim \|\nabla_\Ghsm \ehm \|_{L^p(\Ghsm)}
	\quad\forall p\in[1,\infty]
	.
\end{align}

For later use, we define the orthogonal projectors
\[
N_*^m:=n_*^m(n_*^m)^\top,
\qquad
T_*^m:=I-N_*^m,
\qquad
\hat N_{h,*}^m:=\hat n_{h,*}^m(\hat n_{h,*}^m)^\top,
\qquad
\hat T_{h,*}^m:=I-\hat N_{h,*}^m,
\]
and
\[
\bar N_{h,*}^m
:=\frac{\bar n_{h,*}^m}{|\bar n_{h,*}^m|}
\left(\frac{\bar n_{h,*}^m}{|\bar n_{h,*}^m|}\right)^\top,
\qquad
\bar T_{h,*}^m:=I-\bar N_{h,*}^m 
\qquad
\bar N_{h}^m
:=\frac{\bar n_{h}^m}{|\bar n_{h}^m|}
\left(\frac{\bar n_{h}^m}{|\bar n_{h}^m|}\right)^\top,
\qquad
\bar T_{h}^m:=I-\bar N_{h}^m
.
\]
Note that unnormalized $I-\bar n_h^m(\bar n_h^m)^\top$ in \eqref{eq:Dziuk-stab} is not identical
to the normalized projector $\bar T_h^m$.

\subsection{Super-approximation of $\nbhsm$}
\label{subsec:edge-normal-expansion}

For notational simplicity, some symbols in this section are reused from the other parts of the paper. Their meanings here are local to this section and do not affect their use elsewhere.

The optimal convergence rate relies on a super-approximation property of the reversely weighted normal $\nbhsm$. We establish that property by a local Taylor expansion.
Let $\gamma$ be a sufficiently smooth arc-length parametrization of
the exact curve, centered at $s=0$. Write
\[
x=\gamma(0),
\qquad
t=t(0),
\qquad
n=n(0),
\qquad
\kappa=\kappa(0),
\]
and adopt the Frenet convention
\[
t_s=\kappa n,
\qquad
n_s=-\kappa t,
\qquad
n=\mathcal Rt,
\]
where $\mathcal R$ denotes counterclockwise rotation through $\pi/2$. Since
\[
\gamma'(0)=t,
\qquad
\gamma''(0)=\kappa n,
\qquad
\gamma'''(0)=\kappa_s n-\kappa^2t,
\]
Taylor expansion gives
\begin{align}
	\gamma(a)
	&=
	\gamma(0)
	+a t
	+\frac{a^2}{2}\kappa n
	+\frac{a^3}{6}
	\bigl(\kappa_s n-\kappa^2t\bigr)
	+O(a^4).
	\notag
\end{align}
Consequently, the chord vector $d:=\gamma(a)-\gamma(0)$ satisfies
\begin{align}
	d
	&=
	a t
	+\frac{a^2}{2}\kappa n
	+\frac{a^3}{6}
	\bigl(\kappa_s n-\kappa^2t\bigr)
	+O(a^4).
	\notag
\end{align}
Using $t\cdot n=0$ and $|t|=|n|=1$, we obtain
\begin{align*}
	|d|
	=
	|a|
	\sqrt{\left(1-\frac{\kappa^2a^2}{6}\right)^2
		+a^2\left(\frac{\kappa}{2}+\frac{\kappa_sa}{6}\right)^2
		+O(a^3)}
	=
	|a|\left(1-\frac{\kappa^2a^2}{24}\right)+O(a^4),
\end{align*}
and hence
\begin{align}\label{eq:d-a}
	\big| |d| - |a| \big| \leq C|a|^3.
\end{align}

Define the oriented chord tangent and its normal by
\[
\tau_d:={\rm sign}(a)\frac{d}{|d|},
\qquad
n_d:=\mathcal R\tau_d.
\]
Then
\begin{align}
	\tau_d
	&=
	\left( t
	+\frac{a}{2}\kappa n
	+\frac{a^2}{6}
	\bigl(\kappa_s n-\kappa^2t\bigr)
	+O(a^3)\right)
	\left(1+\frac{\kappa^2a^2}{24}+O(a^3)\right)
	\notag\\
	&=
	t+\frac{\kappa a}{2}n
	+a^2\left(
	\frac{\kappa_s}{6}n
	-\frac{\kappa^2}{8}t
	\right)
	+O(a^3),
	\notag
\end{align}
and therefore
\begin{align}
	n_d
	&=
	n-\frac{\kappa a}{2}t
	-a^2\left(
	\frac{\kappa_s}{6}t
	+\frac{\kappa^2}{8}n
	\right)
	+O(a^3).
	\notag
\end{align}

We now return to the finite element curve $\Ghsm$. Let
$x_j\in\mathcal N(\Ghsm)$ and choose the arc-length coordinate so that
\[
x_j=\gamma(0),
\qquad
x_{j-1}=\gamma(-a_-),
\qquad
x_{j+1}=\gamma(a_+),
\]
where $a_-,a_+>0$ and $a_-,a_+\simeq h$. If $n_-$ and $n_+$ denote the
oriented normals of the left and right chords, respectively, then
\begin{align}
	n_-
	&=
	n+\frac{\kappa a_-}{2}t
	-a_-^2\left(
	\frac{\kappa_s}{6}t
	+\frac{\kappa^2}{8}n
	\right)
	+O(a_-^3),
	\notag\\
	n_+
	&=
	n-\frac{\kappa a_+}{2}t
	-a_+^2\left(
	\frac{\kappa_s}{6}t
	+\frac{\kappa^2}{8}n
	\right)
	+O(a_+^3).
	\notag
\end{align}
Set
\[
\ell_-:=|x_j-x_{j-1}|,
\qquad
\ell_+:=|x_{j+1}-x_j|.
\]
By \eqref{eq:d-a}, $\ell_\pm=a_\pm+O(h^3)$. The reverse weighting in
\eqref{bar-n-hat-n} therefore cancels the first-order tangential terms and results in
\begin{align}
	\nbhsm(x_j)
	&=\frac{\ell_+ n_-+\ell_- n_+}{\ell_-+\ell_+}
	=\frac{a_+ n_-+a_- n_+}{a_-+a_+}+O(h^2)
	=n(x_j)+O(h^2),
	\label{eq:reverse-average-cancellation}
\end{align}
where the constant in $O(h^2)$ depends only on the geometric quantities of the underlying exact curve $\Gamma^m$.

Summarizing the results in Sections \ref{subsec:averaged-normal-basic} and \ref{subsec:edge-normal-expansion}, we have the following lemma.
\begin{lemma}\label{lemma:n_bar_app}
	For $0\leq m\leq l$, the averaged normals satisfy
	\begin{align}
		\| \nbhsm - I_h \nsm \|_{L^\infty(\Ghsm)}
		&\lesssim h^2,
		\notag\\
		\| \nbhm - I_h \nsm \|_{L^2(\Ghsm)}
		&\lesssim h^2 + \|\nabla_\Ghsm \ehm \|_{L^2(\Ghsm)},
		\notag\\
		\| \nbhsm - \hat n_{h,*}^m \|_{L^\infty(\Ghsm)}
		&\lesssim h.
		\label{eq:averaged-normal-approximation}
	\end{align}
\end{lemma}

\begin{proof}
	The first estimate follows from the nodal expansion
	\eqref{eq:reverse-average-cancellation} and the definition of $I_hn_*^m$.
	Combining this estimate with \eqref{normal-intpl-3} gives the second one.
	The third estimate follows from the first estimate and
	\eqref{normal-intpl} by the triangle inequality.
\end{proof}

\subsection{Some geometric identities}
\label{section:geometry}

In the remainder of this section, we will identify all finite element functions as elements either in $S_h(\Ghsm)$ or $S_h(\Ghsm)^2$, via the canonical nodal vector identification described in Section \ref{subsec:parametric-finite-elements}.
We define $X^{m+1}:\Gamma^0\rightarrow\Gamma^{m+1}$ and $Y^{m+1}:= X^{m+1}\circ (X^m)^{-1} :\Gamma^m\rightarrow\Gamma^{m+1}$ to be the exact global and local flow maps along $-H(t)n(t)$, respectively.
Let the interface finite element function $X_{h,*}^{m+1}:\Ghsm\rightarrow \Gamma_{h,*}^{m+1}$ be the interpolation of the local flow, which is uniquely determined by the relation $$X_{h,*}^{m+1}(p) - \hat X_{h,*}^m(p) = Y^{m+1}(p) - {\rm id_\Gm}(p)\quad\forall p\in\mathcal N(\Ghsm)\subset\Gm .$$
Then, it follows that 
\begin{align}
	&X_{h,*}^{m+1} - \hat X_{h,*}^m =I_h \big((Y^{m+1} - {\rm id_\Gm}) \circ a^m|_\Ghsm \big) &&\mbox{on}\,\,\, \Ghsm , \label{eq:X-id1} \\
	&Y^{m+1} - {\rm id}_\Gm = \tau ( {{-H^m n^m }}  +  g^m )  &&\mbox{on}\,\,\,  \Gamma^m, \label{eq:X-id2}
\end{align}
where $g^m$ is a smooth correction from the Taylor expansion, satisfying the following $W^{1,\infty}$ estimate: 
\begin{align}\label{W1infty-g}
	\|g^m\|_{W^{1,\infty}(\Gamma^m)}\le C\tau . 
\end{align}


The local trajectory error and the projection error at time level $m$ are defined as $e_h^m := X_h^m -  X_{h,*}^m$ and $\ehm := X_h^m -  \hat X_{h,*}^m$, respectively.
According to \cite[Eqs. (3.12)--(3.13)]{BL2024}, we have the following nodal relation
\begin{align}\label{eq:geo_rel_1}
	\ehm = I_h\big[ (e_h^m\cdot \nsm)\nsm \big] + r_h^m ,
\end{align}
with $r_h^m := \ehm - I_h\big[ (e_h^m\cdot \nsm)\nsm \big]$ satisfying
\begin{align}\label{eq:geo_rel_2}
	|r_h^m| \lesssim |[I - \nsm\otimes \nsm ] e_h^m|^2
	\quad\mbox{at the nodes of $\Ghsm$}.	
\end{align}
$r_h^m$ can be interpreted as a quadratic remainder of the nodal orthogonal projection due to the presence of curvature.

Therefore, we deduce from \eqref{eq:X-id1}, \eqref{eq:X-id2} and \eqref{eq:geo_rel_1} that
\begin{align}\label{eq:geo_rel_3}
	\begin{aligned}
		X_h^{m+1} - X_h^m 
		&= \eM - \ehm + X_{h,*}^{m+1} - \hat X_{h,*}^m \\
		&= \eM - \ehm + \tau I_h\big((-H^m n^m + g^m) \circ a^m|_\Ghsm \big) ,
	\end{aligned}
\end{align}
at the finite element nodes in $\mathcal N(\Ghsm)$.
This relation helps us convert the numerical displacement $X_h^{m+1} - X_h^m$ to the error displacement $\eM - \ehm$.
Denote by $\delta_\tau X_h^m:=(X_h^{m+1}-X_h^m+\tau I_h\big((H^m n^m) \circ a^m|_\Ghsm \big))/\tau$ and $\delta_\tau \ehm:=(\eM-\ehm)/\tau$, and then \eqref{eq:geo_rel_3} becomes
\begin{align}\label{eq:geo_rel_31}
	\delta_\tau X_h^m = \delta_\tau \ehm +  I_h\big(g^m \circ a^m|_\Ghsm \big) .
\end{align}

The following geometric identities related to the projection error are well-known (cf. \cite[Eqs. (A.15)--(A.17)]{BL2024}): 
If we define $\rho_h^m := I_h(\Nsm (\hat X_{h,*}^{m + 1} -\hat X_{h,*}^{m})) - I_h((Y^{m + 1} - {\rm id}_\Gm)\circ a^m|_\Ghsm)\in [S_h(\Ghsm)]^2$, then at all finite element nodes in $\mathcal N(\Ghsm)$, it holds that
\begin{align}
	\Nsm (\hat X_{h,*}^{m + 1} -\hat X_{h,*}^{m})
	&= (Y^{m + 1} - {\rm id}_\Gm)\circ a^m|_\Ghsm + \rho_h^m  ,
	\label{eq:geo_rel_4}
	\\
	\mbox{where}\,\,\, | \rho_h^m| 
	&\le C_0 \tau^2 + C_0 |T_*^m (\hat X_{h,*}^{m + 1} -\hat X_{h,*}^{m})|^2
	, \label{eq:geo_rel_5}\\
	T_*^m (\hat X_{h,*}^{m + 1} -\hat X_{h,*}^{m})
	&= T_*^m (X_{h}^{m + 1} - X_{h}^{m})
	- T_*^m (N_*^{m+1}\circ\hat X_{h,*}^{m+1}-N_*^{m}\circ\hat X_{h,*}^{m})  \ehM
	. \label{eq:geo_rel_6}
\end{align}

The nodal orthogonality of $\hat e_h^m$ to $\Gamma^m$ implies that its tangential component away from the nodes is still of higher order. This can be quantified via the super-approximation estimates in Appendix \ref{sec:super}.
%
%

\subsection{The main theorem and the induction hypothesis}
\label{subsec:main-result-induction}


We assume that the initial triangulation $\mathcal T_h^0$ of $\Ghso$ is shape regular
\begin{align}\label{P0}
	\kappa_0 \leq C_{\rm sh} ,
\end{align}
 and quasi-uniform: There exist
constants $c_{\rm q},C_{\rm q}>0$, independent of $h$, such that
\begin{equation}
	c_{\rm q}h\leq |K|\leq C_{\rm q}h
	\qquad\forall\,K\in\mathcal T_h^0.
	\label{P1}
\end{equation}
The main convergence result can now be stated precisely.

\begin{theorem}[Convergence of the stabilized Dziuk method]
	\label{thm:main}
	Suppose that the flow map
	$X:\Gamma^0\times[0,T]\rightarrow\mathbb R^2$ of the curve-shortening
	flow and its inverse $X(\cdot,t)^{-1}:\Gamma(t)\rightarrow\Gamma^0$ are
	sufficiently smooth, uniformly for $t\in[0,T]$. Suppose further that the
	initial curve $\Gamma_h^0$ is closed, satisfies \eqref{P0}--\eqref{P1}, and obeys
	\[
	\|\hat e_h^0\|_{L^2(\hat\Gamma_{h,*}^0)}\leq c_0h^2,
	\]
	for a constant $c_0$ independent of $h$. Let $X_h^m$ be the finite element
	solution of \eqref{eq:Dziuk-stab}, with $X_h^0={\rm id}$ on $\Gamma_h^0$.
	For any fixed constants $0<c_1\leq c_2$, there exists $h_0>0$ such that,
	if
	\[
	c_1h^2\leq\tau\leq c_2h^2,
	\qquad h\leq h_0,
	\]
	then the piecewise linear finite element solution satisfies
	\begin{align}
		\label{eq:err_est_1}
		&\max_{1\le m\le \lfloor T/\tau\rfloor}
		\| \hat e_h^{m} \|_{L^2(\Ghsm)}^2
		+
		\sum\limits_{m = 1}^{\lfloor T/\tau\rfloor}
		\tau \| \nabla_{\Ghsm} \hat e_h^{m}\|_{L^2(\Ghsm)}^2
		\le Ch^{4}.
	\end{align}
	Here $C$ is independent of $h$ and $\tau$, but may depend on
	$c_0,c_1,c_2,c_{\rm q},C_{\rm q},C_{\rm sh},T$, and the exact solution.
\end{theorem}

\begin{remark}\label{rmk:thm}\upshape
	The upper step-size bound $\tau\leq c_2h^2$ is used in the
	shape-regularity analysis; see \eqref{eq:X-hat-diff-triangle}. The lower
	bound $\tau\geq c_1h^2$ is needed to control the negative power of
	$\tau$ in \eqref{eq:tan_stab_H1_e}.
\end{remark}

The proof is organized as a continuation argument. Let
$M:=\lfloor T/\tau\rfloor$ and fix $l\in\{0,\ldots,M-1\}$. We assume that
the following estimate holds for every $m=0,\ldots,l$ and prove that it
remains valid at $m=l+1$:
\begin{align}
	&\| \ehm \|_{L^2(\Ghsm)} + \| e_h^m \|_{L^2(\Ghsm)}
	+h^{1/2} \big(\| \ehm \|_{L^\infty(\Ghsm)}
	+ \| e_h^m \|_{L^\infty(\Ghsm)}\big)
	\notag\\
	&\quad
	+ h\big(\| \ehm \|_{H^1(\Ghsm)}
	+ \| e_h^m \|_{H^1(\Ghsm)}\big)
	+ h^{3/2}\big(\| \ehm \|_{W^{1,\infty}(\Ghsm)}
	+ \| e_h^m \|_{W^{1,\infty}(\Ghsm)}\big)
	\leq h^{7/4},
	\label{eq:ind_hypo1}
\end{align}
%
with the convention $e_h^0 := \hat e_h^0$.
When $e_h^m$ is realized on $\Ghsm$ in \eqref{eq:ind_hypo1}, the
realization is understood in the sense of the canonical nodal identification.

In particular, \eqref{eq:ind_hypo1} yields
\begin{align}
	\|\nabla_{\Ghsm}\ehm\|_{L^\infty(\Ghsm)}\leq h^{1/4}. \notag
\end{align}
We define the linearly interpolated curves
\[
\hat\Gamma_{h,\theta}^m
:=({\rm id}+\theta\ehm)(\Ghsm)
=(1-\theta)\Ghsm+\theta\Ghm,
\]
with the transition maps
$({\rm id}+\theta\ehm):\Ghsm\rightarrow\hat\Gamma_{h,\theta}^m$.
For sufficiently small $h$, applying
Lemma~\ref{lemma:norm-equiv} to this family shows that the $L^p$ and
$W^{1,p}$ norms of finite element functions with a common nodal vector are
uniformly equivalent on $\Ghsm$, $\hat\Gamma_{h,\theta}^m$, and $\Ghm$.

Throughout the remainder of the paper, $C$ denotes a generic positive
constant that is independent of $h$, $\tau$, and $m$, but may initially
depend on $\kappa_l$, $T$, and the smooth exact solution. We write
$C_{\kappa_l}$ when this dependence is relevant and use $C_0$ for constants
independent of $\kappa_l$. The notation $A\lesssim B$ means $A\leq CB$.

Lemma~\ref{lemma:n_bar_app}, the smoothness of $n_*^m$, inverse estimates,
	and the induction hypothesis (Eq. \eqref{eq:ind_hypo1}) imply
	\begin{align}\label{eq:n_bar_bbd}
		\begin{split}
			\| \nbhsm \|_{W^{1,\infty}(\Ghsm)}
			&\lesssim 1+h\lesssim 1,
			\\
			\| \nbhm \|_{L^\infty(\Ghsm)}
			&\lesssim 1+h^2
			+h^{-1/2}\|\nabla_\Ghsm\ehm\|_{L^2(\Ghsm)}
			\lesssim 1,
			\\
			\| \nbhm \|_{W^{1,\infty}(\Ghsm)}
			&\lesssim 1+h
			+h^{-3/2}\|\nabla_\Ghsm\ehm\|_{L^2(\Ghsm)}
			\lesssim 1+h^{-3/4}.
		\end{split}
	\end{align}

\section{Consistency analysis}\label{sec:cons_err}

We insert the interpolated exact one-step flow into the discrete scheme and estimate the re-
sulting residual. The time discretization, geometric perturbation, quadrature, and stabilization
contributions are treated separately.

\subsection{Geometry perturbation estimates}

The following are the standard geometry perturbation estimates (cf. \cite[Lemma 5.6]{Kov18}) which are helpful in dealing with the consistency errors.
\begin{lemma}\label{lemma:geo-pert}
	Given $1/p+1/q=1$, the following estimates hold{\rm:}
	\begin{align*}
		\Big|\int_{\Ghsm} f_1f_2 - \int_{\Gm} f_1^\ell f_2^\ell \Big| 
		&\le C_{\kappa_l} h^{2} \|f_1\|_{ L^p(\Ghsm)}\|f_2\|_{L^q(\Ghsm)},\\
		\Big|\int_{\Ghsm}\nabla_{\Ghsm} f_1\cdot\nabla_{\Ghsm} f_2 - \int_{\Gm}\nabla_{\Gm} f_1^\ell\cdot \nabla_{\Gm} f_2^\ell \Big|& \le C_{\kappa_l} h^{2}\|\nabla_{\Ghsm}f_1\|_{ L^p(\Ghsm)}\|\nabla_{\Ghsm}f_2\|_{L^q(\Ghsm)} , \\
		\Big|\int_{\Ghsm}\nabla_{\Ghsm} f_1\cdot f_2 - \int_{\Gm}\nabla_{\Gm} f_1^\ell\cdot f_2^\ell \Big|& \le C_{\kappa_l} h\|\nabla_{\Ghsm}f_1\|_{ L^p(\Ghsm)}\| f_2\|_{L^q(\Ghsm)} . 
	\end{align*}
\end{lemma} 

\subsection{Definition of consistency errors}

The consistency error $\mathscr D^m(\cdot)$ associated with the scheme \eqref{eq:Dziuk-stab} is defined as a linear functional on $[S_h(\Ghsm)]^2$: 
\begin{align}\label{eq:cons_eq}
\mathscr D^m(\phi_h) &:= \int_{\hat\Gamma_{h,*}^{m}}^h \frac{X_{h,*}^{m+1}-{\rm id}}{\tau} \cdot \phi_h
+ \int_{\hat\Gamma_{h,*}^{m}} \nabla_{\hat\Gamma_{h,*}^m} X_{h,*}^{m+1} \cdot \nabla_{\hat\Gamma_{h,*}^m} \phi_h  \notag \\
&\quad- \int_{\hat\Gamma_{h,*}^{m}} \nabla_{\hat\Gamma_{h,*}^m} \hat X_{h,*}^{m} \cdot \nabla_{\hat\Gamma_{h,*}^m} I_h[ (I -\bar n_{h,*}^m (\bar n_{h,*}^m)^\top ) \phi_h] \notag\\
&= \int_{\hat\Gamma_{h,*}^{m}}^h \frac{X_{h,*}^{m+1}-{\rm id}}{\tau} \cdot \phi_h
+ \int_{\Gm} \Hm\nm \cdot \phi_h^\ell  \notag\\
&\quad- \int_{\Gm} \nabla_\Gm {\rm id} \cdot \nabla_\Gm \phi_h^\ell
+ \int_{\hat\Gamma_{h,*}^{m}} \nabla_{\hat\Gamma_{h,*}^m} X_{h,*}^{m+1} \cdot \nabla_{\hat\Gamma_{h,*}^m} \phi_h  \notag\\
&\quad- \int_{\hat\Gamma_{h,*}^{m}} \nabla_{\hat\Gamma_{h,*}^m} \hat X_{h,*}^{m} \cdot \nabla_{\hat\Gamma_{h,*}^m} I_h[ (I -\bar n_{h,*}^m (\bar n_{h,*}^m)^\top ) \phi_h]  \notag\\
&=: \mathscr D_1^m(\phi_h) + \mathscr D_2^m(\phi_h) +\mathscr D_3^m(\phi_h) ,
\end{align}
where we have used the identity 
$\displaystyle \int_{\Gm} \nabla_\Gm {\rm id} \cdot \nabla_\Gm \phi_h^\ell = \int_{\Gm} \Hm\nm \cdot \phi_h^\ell $. 

\subsection{Estimates for consistency errors}

The following lemma gives the estimate for the consistency error defined in \eqref{eq:cons_eq}.
\begin{lemma}\label{proposition-consistency}
The consistency error $\mathscr D^m(\cdot)$ satisfies the following estimate: 
\begin{align}
	| \mathscr D^m(\phi_h)| &\lesssim \tau \| \phi_h \|_{L^2(\Ghsm)} + h^{2} \| \phi_h \|_{H^1(\Ghsm)} 
	\quad\forall\,\phi_h\in [S_h(\Ghsm)]^2 .  \notag
\end{align}
\end{lemma}

\begin{proof}
Using $\eqref{eq:X-id1}$, we decompose the first contribution on the right-hand side of \eqref{eq:cons_eq} into four terms:
\begin{align}
	\mathscr D_1^m(\phi_h) &= \int_{\hat\Gamma_{h,*}^{m}}^h \frac{X_{h,*}^{m+1}-{\rm id}}{\tau} \cdot \phi_h
	+ \int_{\Gm} \Hm\nm \cdot \phi_h^\ell  \notag\\
	&= \int_{\hat\Gamma_{h,*}^{m}}^h I_h\Big( \frac{Y^{m+1} - {\rm id}}{\tau} + H^m n^m \Big)^{-\ell}  \cdot \phi_h \notag\\
	&\quad- \int_{\hat\Gamma_{h,*}^{m}}^h (I_h(H^m n^m)^{-\ell} - (H^m n^m)^{-\ell} ) \cdot \phi_h \notag\\
	&\quad- \Big(\int_{\hat\Gamma_{h,*}^{m}}^h - \int_{\hat\Gamma_{h,*}^{m}} \Big) (H^m n^m)^{-\ell}  \cdot \phi_h \notag\\
	&\quad- \int_{\hat\Gamma_{h,*}^{m}} (H^m n^m)^{-\ell} \cdot \phi_h
	+ \int_{\Gm} H^m n^m \cdot \phi_h^\ell  \notag\\
	&=: \sum_{i=1}^4 \mathscr D_{1i}^m(\phi_h) . \notag
\end{align}
Here $\big(\int^h_{\hat\Gamma_{h,*}^{m}} - \int_{\hat\Gamma_{h,*}^{m}} \big) f = \int^h_{\hat\Gamma_{h,*}^{m}} f - \int_{\hat\Gamma_{h,*}^{m}} f$ abbreviates the quadrature error for any piecewise continuous function $f$ defined on $\hat\Gamma_{h,*}^{m}$.
Relations \eqref{eq:X-id2}--\eqref{W1infty-g} and the nodal identity give, respectively,
\begin{align*}
	| \mathscr D_{11}^m(\phi_h)|  &\lesssim \tau \| \phi_h \|_{L^2(\Ghsm)} ,\\
	| \mathscr D_{12}^m(\phi_h)|  &\lesssim h^2 \| \phi_h \|_{L^2(\Ghsm)} . 
\end{align*}
The superconvergence estimate in Lemma~\ref{lemma:super_conv2} implies
\begin{align*}
	| \mathscr D_{13}^m(\phi_h) | &\lesssim h^{2} \| (H^m n^m)^{-\ell} \|_{H^{2}_h(\Ghsm)} \| \phi_h \|_{H^1(\Ghsm)} 
	\lesssim h^{2}  \| \phi_h \|_{H^1(\Ghsm)} .
\end{align*}
The geometric perturbation estimate in Lemma~\ref{lemma:geo-pert}, together with norm equivalence, bounds the final term:
\begin{align*}
	| \mathscr D_{14}^m(\phi_h)| \lesssim h^{2} \| \phi_h \|_{L^2(\Ghsm)} .
\end{align*}
Combining the estimates of $\mathscr D_{1i}^m(\phi_h)$, $i=1,\dots,4$, yields
\begin{align*}
	| \mathscr D_{1}^m(\phi_h)| \lesssim \tau \| \phi_h \|_{L^2(\Ghsm)} + h^{2} \| \phi_h \|_{H^1(\Ghsm)} .
\end{align*}

Next, decompose $\mathscr D_2^{m}(\phi_h)$ from \eqref{eq:cons_eq} as
\begin{align}
	\mathscr D_2^{m}(\phi_h) 
	&= \int_{\Ghsm} \nabla_{\Ghsm} X_{h,*}^{m+1} \cdot \nabla_{\Ghsm} \phi_h - \int_{\Gm} \nabla_\Gm {\rm id} \cdot \nabla_\Gm \phi_h^\ell \notag\\
	&= \int_{\Ghsm} \nabla_{\Ghsm} (X_{h,*}^{m+1}-\hat X_{h,*}^m) \cdot \nabla_{\Ghsm} \phi_h  \notag\\
	&\quad+ \int_{\Ghsm} \nabla_{\Ghsm} \hat X_{h,*}^m \cdot \nabla_{\Ghsm} \phi_h - \int_{\Gm} \nabla_\Gm (\hat X_{h,*}^m)^\ell \cdot \nabla_\Gm \phi_h^\ell  \notag\\
	&\quad+ \int_{\Gm} \nabla_{\Gm} [(I_ha^m)^\ell - a^m ] \cdot \nabla_{\Gm} \phi_h^\ell \notag\\
	&= \mathscr D_{21}^{m}(\phi_h) + \mathscr D_{22}^{m}(\phi_h) + \mathscr D_{23}^{m}(\phi_h) , \notag
\end{align}
where the penultimate equality uses
$$(\hat X_{h,*}^m)^\ell= (I_ha^m)^\ell \quad\mbox{and}\quad {\rm id} = a^m  \,\,\,\mbox{on}\,\,\, \Gm . $$
Following an argument similar to that in \cite[Lemma 4.3]{BL2024}, with the help of the geometric relations \eqref{eq:X-id1}--\eqref{W1infty-g} and integration by parts, we have
\begin{align*}
	|\mathscr D_{21}^{m}(\phi_h)|
	&\lesssim
	\tau \| \phi_h \|_{L^2(\Ghsm)} .
\end{align*}
Lemma~\ref{lemma:geo-pert} implies
\begin{align*}
	|\mathscr D_{22}^{m}(\phi_h)|
	\lesssim h^{2} \| \phi_h \|_{H^1(\Ghsm)} ,
\end{align*}
and the super-approximation estimate (Lemma~\ref{Lemma-GLW}) gives
\begin{align*}
	|\mathscr D_{23}^{m}(\phi_h)|
	\lesssim h^{2} \| \phi_h \|_{H^1(\Ghsm)} .
\end{align*}
Combining the estimates of $\mathscr D_{2i}^{m}(\phi_h)$, $i=1,2,3$, yields
\begin{align*}
	|\mathscr D_{2}^{m}(\phi_h)|
	\lesssim \tau \| \phi_h \|_{L^2(\Ghsm)} + h^{2} \| \phi_h \|_{H^1(\Ghsm)} . 
\end{align*}

It remains to estimate the stabilization contribution. Decompose $\mathscr D_3^m$ as
\begin{align}
	\mathscr D_3^m(\phi_h)
	&=
	-
	\int_{\hat\Gamma_{h,*}^{m}} \nabla_{\hat\Gamma_{h,*}^m} I_h a^m \cdot \nabla_{\hat\Gamma_{h,*}^m} I_h[ (I -\bar n_{h,*}^m (\bar n_{h,*}^m)^\top ) \phi_h]  \notag\\
	&=
	- \int_\Gm \nabla_\Gm {\rm id}_\Gm \cdot \nabla_\Gm I_h[ (I -\bar n_{h,*}^m (\bar n_{h,*}^m)^\top ) \phi_h]^{\ell}  \notag\\
	&\quad
	-
	\bigg(\int_{\hat\Gamma_{h,*}^{m}} \nabla_{\hat\Gamma_{h,*}^m}  ({\rm id}_\Gm)^{-\ell} \cdot \nabla_{\hat\Gamma_{h,*}^m} I_h[ (I -\bar n_{h,*}^m (\bar n_{h,*}^m)^\top ) \phi_h] 
	\notag\\
	&\quad-
	\int_\Gm \nabla_\Gm {\rm id}_\Gm \cdot \nabla_\Gm I_h[ (I -\bar n_{h,*}^m (\bar n_{h,*}^m)^\top ) \phi_h]^{\ell}
	\bigg)
	\notag\\
	&\quad
	+
	\int_{\hat\Gamma_{h,*}^{m}} \nabla_{\hat\Gamma_{h,*}^m} (1-I_h) a^m \cdot \nabla_{\hat\Gamma_{h,*}^m} I_h[ (I -\bar n_{h,*}^m (\bar n_{h,*}^m)^\top ) \phi_h]  \notag\\
	&=: \mathscr D_{31}^m(\phi_h) +\mathscr D_{32}^m(\phi_h)+\mathscr D_{33}^m(\phi_h) . \notag
\end{align}
Using integration by parts, we have
\begin{align}
	|\mathscr D_{31}^m(\phi_h)|
	&= 
	\Big|\int_\Gm H^m n^m \cdot I_h[ (I -\bar n_{h,*}^m (\bar n_{h,*}^m)^\top ) \phi_h]^{\ell} \Big|
	\notag\\
	&= 
	\Big|\int_\Gm H^m n^m \cdot I_h [(\Tbhsm-\Tsm) \phi_h]^{\ell}
	-
	\int_\Gm H^m n^m \cdot (1-I_h) (\Tsm \phi_h)^{\ell}
	\notag\\
	&\quad
	+
	\int_\Gm H^m n^m \cdot I_h\Big[ \Big(\frac{\nbhsm (\nbhsm)^\top }{|\nbhsm|^2} -\nbhsm (\nbhsm)^\top \Big) \phi_h\Big]^{\ell} \Big|
	\notag\\
	&\lesssim
	h^2 \| \phi_h \|_{H^1(\Ghsm)} . \notag
\end{align}
The geometric perturbation estimate (Lemma~\ref{lemma:geo-pert}) and the boundedness \eqref{eq:n_bar_bbd} give
\begin{align}
	|\mathscr D_{32}^m(\phi_h)|
	&\lesssim
	h^2 \| \phi_h \|_{H^1(\Ghsm)} . \notag
\end{align}
Local integration by parts on each element of $\Ghsm$ and the vanishment of $(1-I_h)a^m$ on $\mathcal N(\Ghsm)$ show that $\mathscr D_{33}^m$ vanishes:
\begin{align}
	\mathscr D_{33}^m(\phi_h) = 0. \notag
\end{align}

Combining the estimates of $\mathscr D_1^m(\phi_h)$, $\mathscr D_2^m(\phi_h)$ and $\mathscr D_3^m(\phi_h)$ proves Lemma~\ref{proposition-consistency}. This residual bound is the consistency input for the energy argument below.
\hfill\end{proof}

\section{Stability analysis}
\label{sec:stab}

\subsection{The error equation and the discrete $H^1$ parabolicity decomposition}\label{sec:J_stab}

Subtracting the consistency relation \eqref{eq:cons_eq} from the numerical scheme \eqref{eq:Dziuk-stab} gives the error equation
\begin{align}\label{eq:err_eq1}
	&\int_{\Gamma_{h}^{m}}^h \frac{X_h^{m+1}- X_h^m}{\tau} \cdot \phi_h 
	- \int_{\hat\Gamma_{h, *}^{m}}^h \frac{X_{h, *}^{m+1}- \hat X_{h, *}^m}{\tau}   \cdot \phi_h \notag\\
	&\quad+ \int_{\Gamma_h^m} \nabla_{\Gamma_h^m} X_h^{m+1} \cdot  \nabla_{\Gamma_h^m}  \phi_h  
	-\int_{\hat\Gamma_{h,*}^m} \nabla_{\hat\Gamma_{h,*}^m} X_{h,*}^{m+1} \cdot  \nabla_{\hat\Gamma_{h,*}^m}  \phi_{h}    \notag\\
	&\quad- \int_{\Gamma_h^m} \nabla_{\Gamma_h^m} X_h^{m} \cdot  \nabla_{\Gamma_h^m} I_h [( I - \nbhm  (\nbhm)^\top ) \phi_h ]
	+ \int_{\hat\Gamma_{h,*}^m} \nabla_{\hat\Gamma_{h,*}^m} \hat X_{h,*}^{m} \cdot  \nabla_{\hat\Gamma_{h,*}^m} I_h [(I - \nbhsm (\nbhsm)^\top ) \phi_{h} ]   \notag\\
	&= -\mathscr D^m(\phi_h) .
\end{align} 
The first line on the left-hand side decomposes as
\begin{align}\label{eq:mass_diff}
	&\int_{\Gamma_{h}^{m}}^h \frac{X_h^{m+1}- X_h^m}{\tau}  \cdot \phi_h 
	- 
	\int_{\hat\Gamma_{h, *}^{m}}^h \frac{X_{h, *}^{m+1}- \hat X_{h, *}^m}{\tau}  \cdot \phi_h  \notag\\
	=&\!: \int_{\hat\Gamma_{h, *}^{m}}^h \frac{e_h^{m+1}-\hat e_h^m}{\tau}  \cdot \phi_h + \mathscr J^m(\phi_h) ,
\end{align} 
with
\begin{align}\label{def-Jm-phi}
	\mathscr J^m(\phi_h) &=  \int_{\Ghm}^h \frac{X_{h}^{m+1}- X_{h}^m}{\tau}  \cdot \phi_h  -  \int_{\hat\Gamma_{h, *}^{m}}^h \frac{X_{h}^{m+1}- X_{h}^m}{\tau}  \cdot \phi_h 
	.
\end{align}

Following \cite[Section 5.2]{BL2024}, the second and third lines on the left-hand side of \eqref{eq:err_eq1} can be written as an $H^1$ bilinear form plus lower-order terms. For $\R^2$-valued functions $u$ and $v$ 
on a piecewise smooth curve $\Sigma$, we define
\begin{align}
	\mathscr A_{\Sigma}(u,v)
	&:=\int_{\Sigma}\nabla_{\Sigma}u\cdot\nabla_{\Sigma}v,\notag\\
	\mathscr A^{N}_{\Sigma}(u,v)
	&:=\int_{\Sigma}
	[(\nabla_{\Sigma}u)n]\cdot[(\nabla_{\Sigma}v)n],\notag\\
	\mathscr A^{T}_{\Sigma}(u,v)
	&:=\int_{\Sigma}
	{\rm tr}\!\big((\nabla_{\Sigma}u)(I-nn^\top)
	(\nabla_{\Sigma}v)^\top\big),\notag\\
	\mathscr B_{\Sigma}(u,v)
	&:=\int_{\Sigma}
	(\nabla_{\Sigma}\cdot u)(\nabla_{\Sigma}\cdot v)
	-{\rm tr}(\nabla_{\Sigma}u\nabla_{\Sigma}v). \notag
\end{align}
Thus $\mathscr A_\Sigma=\mathscr A_\Sigma^{N}
+\mathscr A_\Sigma^{T}$.
We define $\ud_i u := (\nabla_\Sigma u)_i $, $i=1,2$.
Using the Einstein summation convention, with
\begin{equation*}
	(D_\Sigma u)_{rl}:=-\ud_lu_r-\ud_ru_l+\delta_{rl}\ud_j u_j ,
\end{equation*}
the following identity holds (see \cite[Eq. (5.8)]{BL2024})
\begin{align}\label{eq:pdf-DGamma-identity}
	\int_\Sigma\nabla_\Sigma{\rm id}\cdot(D_\Sigma u)\nabla_\Sigma v
	&=-\mathscr A_\Sigma^{T}(u,v)+\mathscr B_\Sigma(u,v).
\end{align}

Using the fundamental theorem of calculus and \eqref{eq:pdf-DGamma-identity} (cf. \cite[Eq. (5.10)]{BL2024}), we can obtain the following crucial discrete $H^1$ parabolicity decomposition for the second line of the left-hand side of \eqref{eq:err_eq1}:
\begin{align}\label{eq:full_para}
	&\int_{\Gamma_h^m} \nabla_{\Gamma_h^m} X_h^{m+1} \cdot  \nabla_{\Gamma_h^m}  \phi_h  
	-\int_{\hat\Gamma_{h,*}^m} \nabla_{\hat\Gamma_{h,*}^m} X_{h,*}^{m+1} \cdot  \nabla_{\hat\Gamma_{h,*}^m}  \phi_{h}   \notag\\
	&= \mathscr A_{h, *}^N(e_h^{m+1},\phi_{h}) + \mathscr A_{h, *}^T(e_h^{m+1} - \hat{e}_h^m,\phi_{h}) + \mathscr B^m(\hat{e}_h^m,\phi_{h}) +\mathscr K^m(\phi_{h}) ,
\end{align}
where, for brevity, we use the notation
\begin{align}\label{def-As-AGs}
	\mathscr A_{h,*}^N(u_h,v_h) 
	:\!\!&= \mathscr A_{\Ghsm}^N(u_h,v_h) 
	\quad\mbox{and}\quad
	\mathscr A_{h,*}^T(u_h,v_h) 
	:= \mathscr A_{\Ghsm}^T(u_h,v_h) , \\
	\mathscr A_{h,*}(u_h,v_h) :\!\!&= \mathscr A_{h,*}^N(u_h,v_h)  + \mathscr A_{h,*}^T(u_h,v_h)  
	\quad\mbox{and}\quad
	\mathscr B^m(u_h,v_h) 
	= \mathscr B_{\Gm}(u_h^\ell, v_h^\ell) \label{def-Bm} \\
	\mathscr K^m(\phi_{h}) 
	&= \int_0^1 \big[ \mathscr A_{\hat\Gamma_{h,\theta}^m}^N(e_h^{m+1},\phi_{h}) - \mathscr A_{\Ghsm}^N(e_h^{m+1},\phi_{h}) \big] \d\theta  \notag\\
	&\quad + \int_0^1 
	\big[ \mathscr A_{\hat\Gamma_{h,\theta}^m}^T(e_h^{m+1} - \hat e_h^m,\phi_{h}) - \mathscr A_{\Ghsm}^T(e_h^{m+1} - \hat e_h^m,\phi_{h}) \big] \d\theta   \notag\\
	&\quad+ \int_0^1 \big[ \mathscr B_{\hat\Gamma_{h,\theta}^m}(\hat e_h^m,\phi_{h})  - \mathscr B_{\Ghsm}(\hat e_h^m,\phi_{h}) \big] \d\theta\notag\\
	&\quad + \mathscr B_{\Ghsm}(\hat e_h^m,\phi_{h}) -\mathscr B_{\Gm}(\hat e_h^m,\phi_{h}) \notag\\
	&\quad+\int_0^1 \int_{ \hat\Gamma_{h,\theta}^m}  \nabla_{\hat\Gamma_{h,\theta}^m} (X_{h, \theta}^{m+1}-\hat X_{h, \theta}^{m}) \cdot D_{\hat\Gamma_{h,\theta}^m} \hat e_{h}^{m} \nabla_{\hat\Gamma_{h,\theta}^m} \phi_{h}\d\theta ,
\end{align}
where $\hat{X}_{h, \theta}^{m} := (1 - \theta)\hat X_{h, *}^{m} + \theta X_{h}^{m}$ and $X_{h, \theta}^{m+1} := (1 - \theta)X_{h, *}^{m+1} + \theta X_{h}^{m+1}$ are the linear transports.

The stabilization term can be decomposed similarly as
\begin{align}\label{terms-5-6}
	&\quad- \int_{\Gamma_h^m} \nabla_{\Gamma_h^m} X_h^{m} \cdot  \nabla_{\Gamma_h^m} I_h [(I - \nbhm(\nbhm)^\top ) \phi_h ]
	+ \int_{\hat\Gamma_{h,*}^m} \nabla_{\hat\Gamma_{h,*}^m} \hat X_{h,*}^{m} \cdot  \nabla_{\hat\Gamma_{h,*}^m} I_h [(I  - \nbhsm (\nbhsm)^\top)  \phi_{h} ]   \notag\\
	&= -\int_{\Gamma_h^m} \nabla_{\Gamma_h^m} X_h^{m} \cdot  \nabla_{\Gamma_h^m} I_h [(I - \nbhm(\nbhm)^\top - \Tbhm) \phi_h ]
	\notag\\
	&\quad+ \int_{\hat\Gamma_{h,*}^m} \nabla_{\hat\Gamma_{h,*}^m} \hat X_{h,*}^{m} \cdot  \nabla_{\hat\Gamma_{h,*}^m} I_h[ (I  - \nbhsm (\nbhsm)^\top - \Tbhsm)  \phi_{h} ]   \notag\\
	&\quad- \int_{\Gamma_h^m} \nabla_{\Gamma_h^m} X_h^{m} \cdot  \nabla_{\Gamma_h^m} I_h [(\Tbhm - \Tbhsm) \phi_h ]
	\notag\\
	&\quad- \int_{\Gamma_h^m} \nabla_{\Gamma_h^m} X_h^{m} \cdot  \nabla_{\Gamma_h^m} I_h (\Tbhsm \phi_h )
	+ \int_{\hat\Gamma_{h,*}^m} \nabla_{\hat\Gamma_{h,*}^m} \hat X_{h,*}^{m} \cdot  \nabla_{\hat\Gamma_{h,*}^m} I_h(\Tbhsm  \phi_{h} )   \notag\\
	&=: \mathscr F_1^m(\phi_h) + \mathscr F_2^m(\phi_h) + \mathscr F_3^m(\phi_h) \notag\\
	&\quad - \mathscr A_{h, *}^N(\hat e_h^{m}, I_h \bar T_{h,*}^m \phi_{h}) - \mathscr B^m(\hat e_h^{m}, I_h \bar T_{h,*}^m \phi_{h}) - \mathscr Q^m(I_h \bar T_{h,*}^m \phi_{h})  ,
\end{align}
where in the last identity we have applied the discrete $H^1$ parabolicity decomposition (cf. \cite[Eq. (4.20)]{BL2025}).
Here $\Tbhm= I -  \nbhm(\nbhm)^\top/|\nbhm|^2$ and $\Tbhsm= I -  \nbhsm(\nbhsm)^\top/|\nbhsm|^2$ are the normalized averaged projectors, and the high-order geometric perturbation error $\mathscr Q^m(\cdot)$ is defined as
\begin{align*}
	\mathscr Q^m(\phi_{h}) 
	:= \int_0^1 \big[ \mathscr A_{\hat\Gamma_{h,\theta}^m}^N(\hat e_h^{m},\phi_{h}) - \mathscr A_{\Ghsm}^N(\hat e_h^{m},\phi_{h}) \big] \d\theta 
	&+ \int_0^1 \big[ \mathscr B_{\hat\Gamma_{h,\theta}^m}(\hat e_h^m,\phi_{h})  - \mathscr B_{\Ghsm}(\hat e_h^m,\phi_{h}) \big] \d\theta\notag\\
	& + \mathscr B_{\Ghsm}(\hat e_h^m,\phi_{h}) -\mathscr B_{\Gm}(\hat e_h^m,\phi_{h}) .
\end{align*}

Substituting \eqref{eq:mass_diff}, \eqref{eq:full_para} and \eqref{terms-5-6} into \eqref{eq:err_eq1} yields the following form of the error equation:
\begin{align}\label{eq:err_eq2}
	&\int_{\hat\Gamma_{h, *}^{m}}^h \frac{e_h^{m+1}-\hat e_h^m}{\tau} \cdot  \phi_h  + \mathscr J^m(\phi_h) \notag\\
	&\quad+ \mathscr A_{h, *}^N(e_h^{m+1},\phi_{h}) + \mathscr A_{h, *}^T(e_h^{m+1} - \hat{e}_h^m,\phi_{h}) + \mathscr B^m(\hat{e}_h^m,\phi_{h}) +\mathscr K^m(\phi_{h})   \notag\\
	&\quad+ \sum_{i = 1}^3 \mathscr F_i^m(\phi_h) - \mathscr A_{h, *}^N(\hat e_h^{m}, I_h \bar T_{h,*}^m \phi_{h}) - \mathscr B^m(\hat e_h^{m}, I_h \bar T_{h,*}^m \phi_{h})  - \mathscr Q^m(I_h \bar T_{h,*}^m \phi_{h}) \notag\\
	&= -\mathscr D^m(\phi_h) .
\end{align} 
The discrete $H^1$ parabolicity of the error equation \eqref{eq:err_eq2} becomes clear when choosing $\phi_h = \eM$ and using
\begin{align}
	&\mathscr A_{h, *}^N(e_h^{m+1},\eM) + \mathscr A_{h, *}^T(e_h^{m+1} - \hat{e}_h^m,\eM) 
	\geq \frac{1}{2} \mathscr A_{h,*}(e_h^{m+1},\eM) - \frac{1}{2} \mathscr A_{h,*}^T(\hat e_h^m,\ehm) . \notag
\end{align}
Note that $\frac{1}{2} \mathscr A_{h,*}^T(\hat e_h^m,\ehm)$ is a higher-order smaller term because of the orthogonality between $\ehm$ and $\mathscr A_{h,*}^T(\cdot,\cdot)$ as a result of the super-approximation (Lemma~\ref{lemma:AT-sup}).

\subsection{Bilinear error estimates}

The following bilinear error estimates are standard (cf. \cite[Lemma 4.2]{BL2024}), which can be proved by a fundamental theorem of calculus argument together with the norm equivalence.
\begin{lemma}\label{lemma:e-blinear}
	Assuming the induction hypothesis \eqref{eq:ind_hypo1},
	for all finite element functions $f_h,g_h\in S_h(\Ghsm)$ and $1/p+1/q+1/r=1$, it holds that
	\begin{align*}
		\Big|\int_{\Ghm} f_h g_h - \int_{\Ghsm} f_h g_h \Big| 
		&\lesssim
		\| \nabla_\Ghsm \ehm \|_{L^p(\Ghsm)} \|f_h\|_{ L^q(\Ghsm)}
		\|g_h\|_{L^r(\Ghsm)} ,
		\\
		\Big|\int_{\Ghm} \nabla_\Ghm f_h g_h - \int_{\Ghsm} \nabla_\Ghsm f_h g_h \Big| 
		&\lesssim
		\| \nabla_\Ghsm \ehm \|_{L^p(\Ghsm)} \| \nabla_\Ghsm f_h\|_{ L^q(\Ghsm)}
		\|g_h\|_{L^r(\Ghsm)}
		,
	\end{align*}
	and
	\begin{align*}
		&\Big|\int_{\Ghm}\nabla_{\Ghm} f_h\cdot\nabla_{\Ghm} g_h - \int_{\Ghsm}\nabla_{\Ghsm} f_h \cdot \nabla_{\Ghsm} g_h \Big| 
		\notag\\
		&
		\qquad
		\lesssim
		\| \nabla_\Ghsm \ehm \|_{L^p(\Ghsm)} \| \nabla_\Ghsm f_h\|_{ L^q(\Ghsm)}
		\| \nabla_\Ghsm g_h\|_{L^r(\Ghsm)}
		.
	\end{align*}
	Moreover, \eqref{normal-intpl-2}--\eqref{normal-intpl-3} imply
	\begin{align}
		\Big| \int_{\Ghsm} (\nhm-\nhsm) f_h g_h \Big| 
		+
		\Big| \int_{\Ghsm} (\nbhm-\nbhsm) f_h g_h \Big| 
		\lesssim
		\| \nabla_\Ghsm \ehm \|_{L^p(\Ghsm)} \|f_h\|_{ L^q(\Ghsm)}
		\|g_h\|_{L^r(\Ghsm)} . \notag
	\end{align}
\end{lemma}

\subsection{Estimates for linear and bilinear forms}\label{sec:lin-bilin-est}


First we note that all $\mathscr B$-related terms vanish on curves, since the tangential
gradient has rank one and
\[
\operatorname{tr}(\nabla_\Gamma u\,\nabla_\Gamma v)
=
(\nabla_\Gamma\cdot u)(\nabla_\Gamma\cdot v).
\]

The fundamental theorem of calculus rewrites the functional $\mathscr J^m(\phi_h)$ defined in \eqref{def-Jm-phi} as
\begin{align}
	\mathscr J^m(\phi_h)
	&=\int_{\hat\Gamma_{h,\theta}^m }^h \frac{X_{h}^{m+1}-X_{h}^{m}}{\tau} \cdot \phi_{h}
	\bigg|_{\theta=0}^{\theta=1}
	\notag\\
	&
	= \int_0^1 \int_{\hat\Gamma_{h,\theta}^m }^h \frac{X_{h}^{m+1}-X_{h}^{m}}{\tau} \cdot \phi_{h} (\nabla_{\hat\Gamma_{h,\theta}^m } \cdot \hat e_{h}^m) \d\theta 
	\quad\mbox{(Lemma~\ref{lemma:ud}, item 6)} . \notag
\end{align}
Lemma~\ref{lemma:e-blinear} directly implies
\begin{align}\label{eq:J_est}
	|\mathscr J^m(\phi_h)| 
	&\lesssim
	(1+\| \delta_\tau \ehm \|_{L^\infty(\Ghsm)}) \| \nabla_{\Ghsm} \ehm \|_{L^2(\Ghsm)} \| \phi_h \|_{L^2(\Ghsm)} 
	.
\end{align}
For a tangentially projected test function $I_h\Tsm\phi_h$, mass lumping yields the corresponding directional estimate. In particular, it aligns $\delta_\tau\ehm$ on the right-hand side along the same projection:
\begin{align}\label{eq:J_est-tan}
	|\mathscr J^m(I_h\Tsm\phi_h)| 
	\lesssim
	(1+\|  I_h\Tsm \delta_\tau \ehm \|_{L^\infty(\Ghsm)}) \| \nabla_{\Ghsm} \ehm \|_{L^2(\Ghsm)} \| I_h\Tsm\phi_h \|_{L^2(\Ghsm)} 
	.
\end{align}
Using integration by parts and the orthogonality, we can show (cf. Lemma~\ref{lemma:AT-sup})
\begin{align}
	| \mathscr A_{h, *}^T(\hat{e}_h^m,\phi_h) | 
	&\lesssim \min\Big\{\| \ehm \|_{L^2({\Ghsm})}  \| \phi_h \|_{H^1({\Ghsm})} , \| \ehm \|_{H^1({\Ghsm})}  \| \phi_h \|_{L^2({\Ghsm})}\Big\} \label{eq:AT_em_em}.
\end{align}
As in the derivation of \eqref{eq:AT_em_em}, we also have
\begin{align*}
	|\mathscr A_{h, *}^N(\hat e_h^{m}, I_h \bar T_{h,*}^m \phi_{h})| 
	&\leq 
	| \mathscr A_{h, *}^N(\hat e_h^{m}, I_h \Tsm \phi_{h}) | + | \mathscr A_{h, *}^N(\hat e_h^{m}, I_h (\bar T_{h,*}^m - \Tsm) \phi_{h}) | 
	\notag\\
	&\lesssim 
	\min\Big\{\| \ehm \|_{L^2({\Ghsm})}  \| I_h \bar T_{h,*}^m \phi_h \|_{H^1({\Ghsm})} , \| \ehm \|_{H^1({\Ghsm})}  \| I_h \bar T_{h,*}^m \phi_h \|_{L^2({\Ghsm})}\Big\}
	\notag\\
	&\quad+ h \| \nabla_\Ghsm \ehm \|_{L^2(\Ghsm)} \|  \phi_h \|_{L^2(\Ghsm)}  .
\end{align*}
Another direct application of Lemma~\ref{lemma:e-blinear} yields
\begin{align}
	\label{eq:K_est}
	| \mathscr K^m(\phi_h) | 
	&\lesssim (\| \nabla_{\Ghsm} \ehm \|_{L^2(\Ghsm)} + \| \nabla_{\Ghsm} e_h^{m+1} \|_{L^2(\Ghsm)}) \| \nabla_{\Ghsm} \ehm \|_{L^\infty(\Ghsm)} 
	 \| \nabla_{\Ghsm} \phi_h \|_{L^2(\Ghsm)}  \notag\\
	&\quad+ \tau (1 +  \| \nabla_{\Ghsm} \delta_\tau \ehm \|_{L^2(\Ghsm)}) \| \nabla_{\Ghsm} \hat e_{h}^{m} \|_{L^\infty(\Ghsm)} \| \nabla_{\Ghsm} \phi_h \|_{L^2(\Ghsm)}
	\notag\\
	&\lesssim \| \nabla_{\Ghsm} \ehm \|_{L^\infty(\Ghsm)} \| \nabla_{\Ghsm} \ehm \|_{L^2(\Ghsm)} 
	\| \nabla_{\Ghsm} \phi_h \|_{L^2(\Ghsm)}  \notag\\
	&\quad+ \tau (1 +  \| \nabla_{\Ghsm} \delta_\tau \ehm \|_{L^2(\Ghsm)}) \| \nabla_{\Ghsm} \hat e_{h}^{m} \|_{L^\infty(\Ghsm)} \| \nabla_{\Ghsm} \phi_h \|_{L^2(\Ghsm)}
\end{align}
and
\begin{align}
	| \mathscr Q^m(I_h \bar T_{h,*}^m \phi_{h}) |
	&\lesssim \| \nabla_{\Ghsm} \ehm \|_{L^\infty(\Ghsm)} \| \nabla_{\Ghsm} \ehm \|_{L^2(\Ghsm)}  \| \nabla_{\Ghsm} I_h \bar T_{h,*}^m \phi_h \|_{L^2(\Ghsm)} 
	\label{eq:Q_est}
	,
\end{align}
where we have used 
\begin{align}
	\| \nabla_{\hat\Gamma_{h,\theta}^m} (X_{h, \theta}^{m+1}-\hat X_{h, \theta}^{m}) \|_{L^2(\Ghsm)}
	&\lesssim  \tau (1 +  \| \nabla_{\Ghsm} \delta_\tau \ehm \|_{L^2(\Ghsm)}) ,
	\notag\\
	\|\nabla_{\Ghsm}e_h^{m+1}\|_{L^2(\Ghsm)}
	&\lesssim
	\tau
	+
	\|\nabla_{\Ghsm}\hat e_h^m\|_{L^2(\Ghsm)} + \tau\|\nabla_{\Ghsm}\delta_\tau\ehm\|_{L^2(\Ghsm)} , \notag
\end{align}
 in the derivation of \eqref{eq:K_est}.

We next estimate the remaining terms $\mathscr F_i^m(\phi_h),i=1,2,3$, in \eqref{eq:err_eq2}.
Using \eqref{eq:nhs_bar_len} and Lemma~\ref{lemma:n_bar_app}, for any node $p\in\mathcal N(\Ghsm)$, we have
\begin{align*}
	\big| |\bar n_h^m(p)| - 1 \big|
	&\lesssim 
	|  n_h^m(p+)  - I_hn_*^m(p) |^2 + |  n_h^m(p-)  - I_hn_*^m(p) |^2 \\
	&\lesssim 
	|\nabla_\Ghsm \ehm(p+)|^2+|\nabla_\Ghsm \ehm(p-)|^2+ h^2 ,
\end{align*}
whereas the arithmetically averaged conormal vector has a consistency error of one order lower:
\begin{align*}
	| \mu_h^m(p+) + \mu_h^m(p-) |
	&\lesssim 
	|  n_h^m(p+)  - I_hn_*^m(p) | + |  n_h^m(p-)  - I_hn_*^m(p) | \\
	&\lesssim 
	|\nabla_\Ghsm \ehm(p+)|+|\nabla_\Ghsm \ehm(p-)|+ h .
\end{align*}
Using integration by parts, we estimate $\mathscr F_1^m(\phi_h)$ as follows:
\begin{align}\label{Estimate-F1m}
	|\mathscr F_1^m(\phi_h) | &= \Big| \int_{\Gamma_h^m} \nabla_{\Gamma_h^m} X_h^{m} \cdot  \nabla_{\Gamma_h^m}  I_h [(I - \bar n_h^m (\bar n_h^m)^\top - \bar T_h^m) \phi_h ] \Big|  
	\notag\\
	&= \Big| - \int_{\Gamma_h^m} \Delta_{\Gamma_h^m} X_h^{m} \cdot  I_h [(I - \bar n_h^m (\bar n_h^m)^\top - \bar T_h^m) \phi_h ] 
	\notag\\
	&\quad+ \sum_{p\in\mathcal N(\Ghm)} \Big(\mu_h^m(p+) \cdot (\nabla_{\Gamma_h^m} X_h^{m})(p+) + \mu_h^m(p-) \cdot (\nabla_{\Gamma_h^m} X_h^{m})(p-)\Big)  
	\notag\\
	&\qquad\times
	I_h [(I - \bar n_h^m (\bar n_h^m)^\top - \bar T_h^m) \phi_h ](p) \Big|
	\notag\\
	&=
	\Big|
	\sum_{p\in\mathcal N(\Ghm)} (\mu_h^m(p+) + \mu_h^m(p-))\cdot  [(I - \bar n_h^m (\bar n_h^m)^\top - \bar T_h^m) \phi_h](p)
	\Big|
	\notag\\
	&\lesssim
	\sum_{p\in\mathcal N(\Ghsm)} |\mu_h^m(p+)+\mu_h^m(p-)| \, \big| |\bar n_h^m(p)| - 1 \big| \, |\phi_h(p)|
	\notag\\
	 &\lesssim\sum_{p\in\mathcal N(\Ghsm)} \big(|\nabla_\Ghsm \ehm(p+)|^3+|\nabla_\Ghsm \ehm(p-)|^3 + h^3 \big) \, |\phi_h(p)|
	 \notag\\
	&\lesssim h^{-3/2}\| \nabla_\Ghsm \ehm \|_{L^2(\Ghsm)}^3
	\| \phi_h \|_{L^\infty(\Ghsm)}
	+
	 h^{2} \| \phi_h \|_{L^2(\Ghsm)}
	.
\end{align}
A similar integration-by-parts argument also carries over to $\mathscr F_2^m$ and $\mathscr F_3^m$:
\begin{align}\label{Estimate-F2m}
	|\mathscr F_2^m(\phi_h) | 
	\lesssim
	\sum_{p\in\mathcal N(\Ghsm)}
	|\hat\mu_{h,*}^m(p+) + \hat\mu_{h,*}^m(p-)|
	\cdot
	 \big| |\nbhsm(p)| - 1 \big|  \, |\phi_h(p)|
	\lesssim h^{2} \| \phi_h \|_{L^2(\Ghsm)},
\end{align}
and
\begin{align}\label{Estimate-F3m}
	|\mathscr F_3^m(\phi_h)|&=\Big|\sum_{p\in\mathcal N(\Ghm)} (\mu_h^m(p+) + \mu_h^m(p-))\cdot  (\bar T_h^m(p) - \bar T_{h,*}^m(p)) \phi_h (p) \Big| \notag\\
	&\lesssim
	\sum_{p\in\mathcal N(\Ghsm)} \big(|\nabla_\Ghsm \ehm(p+)|+|\nabla_\Ghsm \ehm(p-)| \big) \big(|\nabla_\Ghsm \ehm(p+)|+|\nabla_\Ghsm \ehm(p-)| + h \big) \, |\phi_h(p)|
	\notag\\
	&\lesssim h^{-1}\| \nabla_\Ghsm \ehm \|_{L^2(\Ghsm)}^2
	\| \phi_h \|_{L^\infty(\Ghsm)}
	+
	\| \nabla_\Ghsm \ehm \|_{L^2(\Ghsm)} \| \phi_h \|_{L^2(\Ghsm)}
	. 
\end{align}

The bounds \eqref{Estimate-F1m}--\eqref{Estimate-F3m}, together with the estimates for $\mathscr J^m$, $\mathscr K^m$, and $\mathscr Q^m$, contain all lower-order terms required in the velocity and energy arguments.

\subsection{Tangential stability for Dziuk's method}\label{sec:tan_stab}

%



Using the definition of $\delta_\tau X_h^m$ in Section~\ref{section:geometry}, the numerical scheme \eqref{eq:Dziuk-stab} becomes
\begin{align}\label{eq:stiff-iden}
	&
	\tau^{-1}\int_{\Gamma_h^m}^h \delta_\tau X_h^m \cdot  \phi_h
	+
	\int_{\Gamma_h^m} \nabla_{\Gamma_h^m} \delta_\tau X_h^m \cdot  \nabla_{\Gamma_h^m}  \phi_h \notag\\
	&= \tau^{-1}\int_{\Gamma_h^m}^h  I_h(H^m n^m) \cdot  \phi_h
	+
	\int_{\Gamma_h^m} \nabla_{\Gamma_h^m} I_h(H^m n^m) \cdot  \nabla_{\Gamma_h^m}  \phi_h
	\notag\\
	&\quad
	-\tau^{-1} \int_{\Gamma_h^m} \nabla_{\Gamma_h^m} X_h^{m} \cdot  \nabla_{\Gamma_h^m} I_h [(\phi_h\cdot\bar n_h^m) \bar n_h^m ]  
	=: \sum_{i = 1}^3 \mathscr L^m_i(\phi_h) .
\end{align}
Since $\Nsm$ and $\Tsm$ are orthogonal and the mass lumping is used here,
\begin{align}
	\mathscr L^m_1(I_h\Tsm\phi_h) =0.  \notag
\end{align}
For the second contribution, decompose
\begin{align}
	&\mathscr L^m_2(I_h\Tsm\phi_h)  
	=  
	\int_{\Gamma_h^m} \nabla_{\Gamma_h^m} I_h(H^m n^m) \cdot  \nabla_{\Gamma_h^m}  I_h\Tsm\phi_h
	\notag\\
	&=  
	\int_{\Gm} \nabla_{\Gm} (H^m n^m) \cdot  \nabla_{\Gm}  (I_h\Tsm\phi_h)^{\ell}
	\notag\\
	&\quad+
	\int_{\Ghsm} \nabla_{\Ghsm} (H^m n^m)^{-\ell} \cdot  \nabla_{\Ghsm}  I_h\Tsm\phi_h
	-
	\int_{\Gm} \nabla_{\Gm} (H^m n^m) \cdot  \nabla_{\Gm}  (I_h\Tsm\phi_h)^{\ell}
	\notag\\
	&\quad-
	\int_{\Ghsm} \nabla_{\Ghsm} (1-I_h)(H^m n^m) \cdot  \nabla_{\Ghsm}  I_h\Tsm\phi_h
	\notag\\
	&\quad+
	\int_{\Gamma_h^m} \nabla_{\Gamma_h^m} I_h(H^m n^m) \cdot  \nabla_{\Gamma_h^m}  I_h\Tsm\phi_h
	-
	\int_{\Ghsm} \nabla_{\Ghsm} I_h(H^m n^m) \cdot  \nabla_{\Ghsm}  I_h\Tsm\phi_h
	\notag\\
	&
	=: \mathscr L^m_{21}(I_h\Tsm\phi_h)  
	+
	\mathscr L^m_{22}(I_h\Tsm\phi_h)  
	+
	\mathscr L^m_{23}(I_h\Tsm\phi_h)  
	+
	\mathscr L^m_{24}(I_h\Tsm\phi_h)   . \notag
\end{align}
Using the integration by parts on the exact curve $\Gm$, we have
\begin{align}
	| \mathscr L^m_{21}(I_h\Tsm\phi_h) | \lesssim 
	\| I_h\Tsm\phi_h  \|_{L^2(\Ghsm)} . \notag
\end{align}
The geometric perturbation estimate (Lemma~\ref{lemma:geo-pert}) gives
\begin{align}
	| \mathscr L^m_{22}(I_h\Tsm\phi_h) | \lesssim  h^2 \| \nabla_\Ghsm I_h\Tsm\phi_h  \|_{L^2(\Ghsm)} . \notag 
\end{align}
Furthermore, local integration by parts yields
\begin{align}
	\mathscr L^m_{23}(I_h\Tsm\phi_h) = 0. \notag
\end{align}
The bilinear error estimate (Lemma~\ref{lemma:e-blinear}) gives
\begin{align}
	| \mathscr L^m_{24}(I_h\Tsm\phi_h) | \lesssim  \| \nabla_\Ghsm \ehm  \|_{L^2(\Ghsm)} \| \nabla_\Ghsm I_h\Tsm\phi_h  \|_{L^2(\Ghsm)} . \notag
\end{align}
Combining these four estimates yields
\begin{align}
	| \mathscr L^m_2(I_h\Tsm\phi_h) | &\lesssim  (h^2 + \| \nabla_\Ghsm \ehm  \|_{L^2(\Ghsm)}) \| \nabla_\Ghsm I_h\Tsm\phi_h  \|_{L^2(\Ghsm)} 
	+
	\| I_h\Tsm\phi_h  \|_{L^2(\Ghsm)} . \notag
\end{align}

From the orthogonality, we know
\begin{align}
	\mathscr L^m_3(I_h\Tsm\phi_h) 
	&= -\tau^{-1} \int_{\Gamma_h^m} \nabla_{\Gamma_h^m} X_h^{m} \cdot  \nabla_{\Gamma_h^m} I_h [(I_h\Tsm\phi_h\cdot\bar n_h^m) \bar n_h^m ]  
	\notag\\
	&= \tau^{-1} \int_{\Gamma_h^m} \nabla_{\Gamma_h^m} X_h^{m} \cdot  \nabla_{\Gamma_h^m} I_h [(I_h(\Tbhm-\Tsm)\phi_h\cdot\bar n_h^m) \bar n_h^m ]  . \notag
\end{align}
Applying integration by parts as in the derivation of \eqref{Estimate-F3m}, we obtain
\begin{align}
	|\mathscr L^m_3(I_h\Tsm\phi_h)| &\lesssim h^{-1}\tau^{-1} (h+\| \nabla_\Ghsm \ehm \|_{L^\infty(\Ghsm)}) (h^2+\| \nabla_\Ghsm \ehm \|_{L^2(\Ghsm)}) \| I_h\Tsm\phi_h \|_{L^2(\Ghsm)} . \notag
\end{align}

Inserting the estimates of $| \mathscr L^m_j(I_h \Tsm\phi_h) |$, $j=1,\dots,3$, into \eqref{eq:stiff-iden}, choosing $\phi_h = I_h \Tsm\delta_\tau X_h^m$, and applying Young's inequality and absorption give
\begin{align}\label{eq:N+TdotT}
	&\tau^{-1}\int_{\Gamma_h^m}^h I_h\Tsm\delta_\tau X_h^m \cdot  I_h\Tsm\delta_\tau X_h^m
	+
	\int_{\Gamma_h^m} \nabla_{\Gamma_h^m} I_h\Tsm \delta_\tau X_h^m \cdot  \nabla_{\Gamma_h^m}  I_h\Tsm\delta_\tau X_h^m
	 \notag\\
	&
	= 
	-
	\int_{\Gamma_h^m} \nabla_{\Gamma_h^m} I_h\Nsm \delta_\tau X_h^m \cdot  \nabla_{\Gamma_h^m}  I_h\Tsm\delta_\tau X_h^m
	+
	\sum_{i=1}^3 \mathscr L^m_i(I_h\Tsm\delta_\tau X_h^m)
	    .
\end{align}
Similar to \cite[Appendix E]{BGV2026} and \cite[Appendix E]{BGV2027}, the following orthogonality cancellation lemma can be proved. 
\begin{lemma}\label{lemma:NT_stab_ref}
	For any $f_h,g_h \in [S_h(\Ghsm)]^2$, we have
	\begin{align}
		&\Big| \int_{\Gamma_h^m} \nabla_{\Gamma_h^m} I_h \Nsm f_h \cdot  \nabla_{\Gamma_h^m} I_h \Tsm g_h \Big|
		\lesssim 
		\| \nabla_\Ghsm\ehm\|_{L^\infty(\Ghsm)} \| \nabla_\Ghsm f_h \|_{L^2(\Ghsm)} \| \nabla_\Ghsm g_h \|_{L^2(\Ghsm)}
		\notag\\
		&\qquad\qquad
		+\min\Big\{
			\| f_h \|_{H^1(\Ghsm)} \| g_h \|_{L^2(\Ghsm)},
		\| f_h \|_{L^2(\Ghsm)} \| g_h \|_{H^1(\Ghsm)}
		\Big\} 
		. \notag
	\end{align}
\end{lemma}
Applying Lemma~\ref{lemma:NT_stab_ref} to \eqref{eq:N+TdotT} with $f_h=I_h \Nsm \delta_\tau X_h^m$ and $g_h = I_h \Tsm \delta_\tau X_h^m$, and using the estimates for the terms $\mathscr L_i^m$, we obtain
\begin{align}\label{eq:tan_stab_H1}
	&
	\tau^{-1/2}\|  I_h \Tsm \delta_\tau X_h^m \|_{L^2(\Ghsm)}  
	+
	\| \nabla_{\Ghsm} I_h \Tsm \delta_\tau X_h^m \|_{L^2(\Ghsm)}  
	\notag\\
	&
	\lesssim 
	\tau^{1/2}
	+
	 \tau^{-1/2}(h^2+\| \nabla_\Ghsm \ehm \|_{L^2(\Ghsm)}+ h^{-3/2} \| \nabla_\Ghsm \ehm \|_{L^2(\Ghsm)}^2 )  \notag\\
	&\quad+(1 + h^{-3/2} \| \nabla_\Ghsm \ehm \|_{L^2(\Ghsm)})  \| I_h \Nsm \delta_\tau X_h^m \|_{L^2(\Ghsm)}  .
\end{align}
Finally, the conversion formula \eqref{eq:geo_rel_31}, which converts $\delta_\tau X_h^m$ into $\delta_\tau \ehm$, together with the bound \eqref{W1infty-g} for $g^m$ gives
\begin{align}\label{eq:tan_stab_H1_e}
	&
	\tau^{-1/2}\|  I_h \Tsm \delta_\tau \ehm \|_{L^2(\Ghsm)}  
	+
	\| \nabla_{\Ghsm} I_h \Tsm \delta_\tau\ehm \|_{L^2(\Ghsm)}  
		\notag\\
	&
	\lesssim 
	\tau^{1/2}
	+
	\tau^{-1/2}(h^2+\| \nabla_\Ghsm \ehm \|_{L^2(\Ghsm)}+ h^{-3/2} \| \nabla_\Ghsm \ehm \|_{L^2(\Ghsm)}^2 )  \notag\\
	&\quad+(1 + h^{-3/2} \| \nabla_\Ghsm \ehm \|_{L^2(\Ghsm)})  \| I_h \Nsm \delta_\tau \ehm \|_{L^2(\Ghsm)}  .
\end{align}

Estimate \eqref{eq:tan_stab_H1_e} controls the $H^1$ semi-norm of the tangential error velocity in terms of the $L^2$ norm of the normal velocity and higher-order smaller geometric errors. 

\subsection{Velocity estimates}\label{sec:bbd_vel} 

We now test the error equation with the full error velocity. Choosing $\phi_h =  \delta_\tau\ehm$ in the error equation \eqref{eq:err_eq2} gives
\begin{align}\label{eq:err1}
	& \int_{\hat\Gamma_{h, *}^{m}}^h \delta_\tau\ehm \cdot \delta_\tau\ehm
	\notag\\
	&= 
	- \mathscr D^m( \delta_\tau\ehm) - \mathscr J^m(\delta_\tau\ehm) - \mathscr B^m(\ehm,  \delta_\tau\ehm) - \mathscr K^m( \delta_\tau\ehm)  \notag\\
	&\quad 
	- \mathscr A_{h, *}^N(e_h^{m+1},  \delta_\tau\ehm ) - \mathscr A_{h, *}^T(e_h^{m+1} - \hat{e}_h^m, \delta_\tau\ehm )  \notag\\
	&\quad- \sum_{i=1}^3 \mathscr F_i^m(\delta_\tau\ehm) + \mathscr A_{h, *}^N(\hat e_h^{m}, I_h\Tbhsm \delta_\tau\ehm) + \mathscr B^m(\hat e_h^{m},  I_h\Tbhsm\delta_\tau\ehm) + \mathscr Q^m( I_h\Tbhsm\delta_\tau\ehm) \notag\\
	&\leq
	- \mathscr D^m( \delta_\tau\ehm) - \mathscr J^m(\delta_\tau\ehm) - \mathscr B^m(\ehm,  \delta_\tau\ehm) - \mathscr K^m( \delta_\tau\ehm)  
	- \mathscr A_{h, *}^N(\ehm,  \delta_\tau\ehm )  \notag\\
	&\quad- \sum_{i=1}^3 \mathscr F_i^m(\delta_\tau\ehm) + \mathscr A_{h, *}^N(\hat e_h^{m},  I_h\Tbhsm \delta_\tau\ehm) + \mathscr B^m(\hat e_h^{m},  I_h\Tbhsm \delta_\tau\ehm) + \mathscr Q^m( I_h\Tbhsm \delta_\tau\ehm)
	,
\end{align}
where the final inequality drops the following two nonpositive terms:
\begin{align*}
- \tau \mathscr A_{h, *}^N(\delta_\tau\ehm, \delta_\tau\ehm)
\qquad
\mbox{and}\qquad - \tau \mathscr A_{h, *}^T( \delta_\tau\ehm, \delta_\tau\ehm) .
\end{align*}

Substituting this bound into the right-hand sides of \eqref{eq:K_est} and \eqref{eq:Q_est} yields
\begin{align}
	&| \mathscr K^m( \delta_\tau\ehm) | + | \mathscr Q^m(I_h \Tbhsm \delta_\tau\ehm) | \notag\\
	&\lesssim (\tau + \| \nabla_{\Ghsm} \ehm \|_{L^\infty(\Ghsm)}) \| \nabla_\Ghsm \ehm \|_{L^2(\Ghsm)} \| \delta_\tau\ehm \|_{H^1(\Ghsm)} \notag\\
	&\quad+ \tau \| \nabla_\Ghsm \ehm \|_{L^\infty(\Ghsm)}\| \nabla_\Ghsm \delta_\tau\ehm \|_{L^2(\Ghsm)} \| \nabla_\Ghsm  \delta_\tau\ehm \|_{L^2(\Ghsm)} 
	. \notag
\end{align}
The definitions of $\mathscr A_{h, *}^N(\cdot,\cdot)$, $\mathscr A_{h, *}^T(\cdot,\cdot)$ and $\mathscr B^m(\cdot,\cdot)$ in \eqref{def-As-AGs}--\eqref{def-Bm} immediately imply
\begin{align}
	|\mathscr A_{h, *}^N(u_h,v_h) | + |\mathscr A_{h, *}^T(u_h,v_h) | + |\mathscr B^m(u_h,v_h) | 
	&\lesssim \| u_h \|_{H^1(\Ghsm)} \| v_h \|_{H^1(\Ghsm)} , \notag
\end{align}
for any $u_h, v_h \in [S_h(\Ghsm)]^2$. Substituting these bounds, together with the estimates of $\mathscr F_i^m, i=1,2,3,$ from \eqref{Estimate-F1m}--\eqref{Estimate-F3m}, into the right-hand side of \eqref{eq:err1} and using the inverse inequality and the induction hypothesis \eqref{eq:ind_hypo1}, yield the full $L^2$ velocity estimate
\begin{align}\label{eq:vel_N}
	\| \delta_\tau\ehm \|_{L^2(\Ghsm)} 
	&\lesssim  h^{-1}(\tau + h^{2}) +  h^{-1} \| \nabla_\Ghsm \ehm \|_{L^2(\Ghsm)}
	\notag\\
	&\quad
	+
	h^{-1}\tau \| \nabla_\Ghsm \ehm \|_{L^\infty(\Ghsm)} \| \nabla_\Ghsm \delta_\tau\ehm \|_{L^2(\Ghsm)}
	\notag\\
	&\lesssim  h^{-1}(\tau + h^{2}) +  h^{-1} \| \nabla_\Ghsm \ehm \|_{L^2(\Ghsm)}
	,
\end{align}
where the last line follows from the absorption of $\delta_\tau\ehm$ into the left-hand side, according to the induction hypothesis \eqref{eq:ind_hypo1} and the step-size condition $\tau\simeq h^2$.
Combining this result with \eqref{eq:tan_stab_H1_e} gives the tangential velocity estimate
\begin{align}
	\| \nabla_{\Ghsm} I_h \Tsm \delta_\tau\ehm \|_{L^2(\Ghsm)}  
		&
		\lesssim 
		\tau^{1/2}
		+
		\tau^{-1/2}(h^2+\| \nabla_\Ghsm \ehm \|_{L^2(\Ghsm)}+ h^{-3/2} \| \nabla_\Ghsm \ehm \|_{L^2(\Ghsm)}^2 )  \notag\\
		&\quad+(1 + h^{-3/2} \| \nabla_\Ghsm \ehm \|_{L^2(\Ghsm)}) 
		\notag\\
	&\qquad\qquad\times
	 \big( h^{-1}(\tau + h^{2}) +  h^{-1} \| \nabla_\Ghsm \ehm \|_{L^2(\Ghsm)}  \big)  
	\notag\\
	&\lesssim 
	\tau^{1/2}
	+
	\tau^{-1/2} h^2
	+
	h^{-1}(\tau+h^2)
	\notag\\
	&\quad
	+
	(\tau^{-1/2}+h^{-1}+(\tau + h^{2})h^{-5/2}) \| \nabla_\Ghsm \ehm \|_{L^2(\Ghsm)}
	\notag\\
	&\quad
	+
	(\tau^{-1/2}h^{-3/2}+h^{-5/2}) \| \nabla_\Ghsm \ehm \|_{L^2(\Ghsm)}^2
	 . \notag
\end{align}
Under the step-size condition $\tau\simeq h^2$, the above estimate reduces to
\begin{align}\label{eq:tan_stab_H1_e2}
	\| \nabla_{\Ghsm} I_h \Tsm \delta_\tau\ehm \|_{L^2(\Ghsm)}  
	&\lesssim 
	h+ h^{-1} \| \nabla_\Ghsm \ehm \|_{L^2(\Ghsm)} + h^{-5/2} \| \nabla_\Ghsm \ehm \|_{L^2(\Ghsm)}^2
	.
\end{align}

Estimates \eqref{eq:vel_N} and \eqref{eq:tan_stab_H1_e2} control, respectively, the full velocity in $L^2$ and the tangential velocity in $H^1$. They are of vital importance in the later shape regularity analysis in Section~\ref{sec:bbd}.

\subsection{Displacement estimates}\label{sec:disc-est}

We next use the velocity estimates to compare $\Ghm, \GhM, \Ghsm, \GhsM$ and $\Gamma_{h,*}^{m+1}$. By Lemma~\ref{lemma:norm-equiv}, equivalence of the $L^p$ and $W^{1,p},p\in[1,\infty],$ norms on these discrete curves follows once their pairwise displacements are sufficiently small in the $W^{1,\infty}$ norm.

The velocity estimates \eqref{eq:vel_N} and \eqref{eq:tan_stab_H1_e2}, the induction hypothesis \eqref{eq:ind_hypo1}, and the step-size condition $\tau\simeq h^{2}$ give the preliminary bounds
\begin{align}\label{eq:a-priori-delta-e}
		\| \nabla_\Ghsm I_h\Tsm \delta_\tau\ehm \|_{L^2(\Ghsm)} 
	\lesssim h^{-1}
	,\qquad
	\| \delta_\tau\ehm \|_{L^2(\Ghsm)} 
	\lesssim h^{-1/4} .
\end{align}
We note that, by construction, $|\ehm(p)| \leq |e_h^m(p)|$ at any finite element node $p\in\mathcal N(\Ghsm)$.
This size relation naturally passes to the mass lumping discrete norm (see Appendix \ref{sec:disc-norm})
\begin{align}\label{eq:hat-stab}
	\| \ehm \|_{L^p_h(\hat\Gamma_{h,*}^{m})} \leq \| e_h^m \|_{L^p_h(\hat\Gamma_{h,*}^{m})}
	\quad\forall p\in[1,\infty]
	.
\end{align}
Using the triangle inequality, \eqref{eq:a-priori-delta-e}, the inverse inequality and the induction hypothesis \eqref{eq:ind_hypo1}, we obtain the following a priori estimates for $\eM$:
\begin{align}
	\| \eM \|_{L^2(\Ghsm)}
	&\leq
	\| \ehm \|_{L^2(\Ghsm)}
	+
	\tau
	\| \delta_\tau\ehm \|_{L^2(\Ghsm)}
	\lesssim h^{7/4}
	\label{eq:eM-L2-prior} , \\
	\| \eM \|_{L^\infty(\Ghsm)}
	&\leq
	\| \ehm \|_{L^\infty(\Ghsm)}
	+
	\tau
	\| \delta_\tau\ehm \|_{L^\infty(\Ghsm)}
	\lesssim h^{5/4}.
	\label{eq:eM-Linf-prior}
\end{align}
Then the geometric relations \eqref{eq:geo_rel_1}--\eqref{eq:geo_rel_2} imply the a priori estimates for $\ehM$:
\begin{align}
	\| \ehM \|_{L^2(\Ghsm)}
	&\lesssim
	\| \eM \|_{L^2(\Ghsm)}
	+
	\| \eM \|_{L^\infty(\Ghsm)} \| \eM \|_{L^2(\Ghsm)}
	\lesssim
	\| \eM \|_{L^2(\Ghsm)}
	\lesssim h^{7/4}
	\label{eq:ehM-L2-prior} ,\\
	\| \ehM \|_{L^\infty(\Ghsm)}
	&\lesssim
	\| \eM \|_{L^\infty(\Ghsm)}
	+
	\| \eM \|_{L^\infty(\Ghsm)}^2
	\lesssim
	\| \eM \|_{L^\infty(\Ghsm)}
	\lesssim h^{5/4}
	\label{eq:ehM-Linf-prior}
	,
\end{align}
and, for the $H^1$ norm,
\begin{align}
	\| \ehM \|_{H^1(\Ghsm)}
	&\lesssim
	\| \eM \|_{H^1(\Ghsm)}
	+
	h^{-1}\| \eM \|_{L^\infty(\Ghsm)} \| \eM \|_{L^2(\Ghsm)}
	\lesssim
	\| \eM \|_{H^1(\Ghsm)}
	\lesssim h^{3/4}
	\label{eq:ehM-H1-prior}
	.
\end{align}

Now we are ready to estimate the displacements. First,
\begin{align}\label{eq:hat_X_s_diff}
	&\| \hat X_{h,*}^{m+1} - \hat X_{h,*}^{m} \|_{L^{\infty}(\Ghsm)} \notag\\ 
	&\le 
	\| \hat X_{h,*}^{m+1} - X_{h}^{m+1} \|_{L^{\infty}(\Ghsm)}
	+\| X_{h}^{m+1} - X_{h}^{m} \|_{L^{\infty}(\Ghsm)}
	+ \| X_{h}^{m} - \hat X_{h,*}^{m} \|_{L^{\infty}(\Ghsm)} \notag\\ 
	&=
	\| \hat e_{h}^{m+1} \|_{L^{\infty}(\Ghsm)}
	+ \|\eM - \ehm - \tau I_h(H^m n^m - g^m)  \|_{L^{ \infty}(\Ghsm)} 
	+\| \hat e_{h}^m \|_{L^{\infty}(\Ghsm)} \notag\\
	&\lesssim
	\| \hat e_{h}^{m+1} \|_{L^{\infty}(\Ghsm)}
	+ \|\eM \|_{L^{ \infty}(\Ghsm)} 
	+\| \hat e_{h}^m \|_{L^{\infty}(\Ghsm)} + \tau  \notag\\
	&\lesssim
	\tau
	+
	 \| \ehm \|_{L^{ \infty}(\Ghsm)} 
	 +
	 \tau\| \delta_\tau\ehm \|_{L^{ \infty}(\Ghsm)}  
	 \lesssim h^{5/4}
	.
\end{align}
The same argument gives the $L^2$ bound:
\begin{align}
	\| \hat X_{h,*}^{m+1} - \hat X_{h,*}^{m} \|_{L^{2}(\Ghsm)}
	&\lesssim
	\tau
	+
	\| \ehm \|_{L^{ 2}(\Ghsm)} 
	+
	\tau\| \delta_\tau\ehm \|_{L^{ 2}(\Ghsm)}  
	. \notag
\end{align}
Using the geometric relation \eqref{eq:X-id1}, we also have
\begin{align}
	\label{eq:hat_X_s_diff3}
\|  X_{h,*}^{m+1} - \hat X_{h,*}^{m} \|_{W^{1,\infty}(\Ghsm)} \lesssim \tau . 
\end{align}

The field $\nsM$ is a smooth extension of $n(\cdot,t_{m+1})$ from $\Gamma^{m+1}$ to a neighborhood of $\Gamma^{m+1}$ that contains $\Gamma^m$ for sufficiently small $\tau$. The gradient of $\nsM$ is bounded uniformly in $m$ and $\tau$. 
Estimate \eqref{eq:hat_X_s_diff3} then gives
\begin{align}
|\nsM-\nsm| 
=\,& 
|n(\hat X_{h,*}^{m+1},t_{m+1}) - n(\hat X_{h,*}^{m},t_{m})| \notag\\
=\, &|n_*^{m+1}(\hat X_{h,*}^{m+1}) -n_*^{m+1}(\hat X_{h,*}^{m}) 
+ n_*^{m+1}(\hat X_{h,*}^{m}) - n_*^{m+1}(X_{h,*}^{m+1}) \notag\\
&+ n(X_{h,*}^{m+1},t_{m+1}) - n(\hat X_{h,*}^{m},t_{m}) | \notag\\
\lesssim\, &|\hat X_{h,*}^{m+1}-\hat X_{h,*}^{m}|
+ |\hat X_{h,*}^{m} - X_{h,*}^{m+1}| 
+ \tau \quad\mbox{at the nodes} \notag\\
\lesssim\, &|\hat X_{h,*}^{m+1}-\hat X_{h,*}^{m}|
+ \tau \quad\mbox{at the nodes} , \notag
\end{align}
where the penultimate inequality uses the smoothness of $n_*^{m+1}$ in a neighborhood of $\Gamma^{m+1}$, and the final inequality uses \eqref{eq:hat_X_s_diff3}. 

Combining \eqref{eq:geo_rel_4}--\eqref{eq:geo_rel_6} with the velocity estimates yields
\begin{align}\label{eq:hat-X-diff-W1inf}
	&\| \hat X_{h,*}^{m+1} - \hat X_{h,*}^{m} \|_{W^{1,\infty}(\Ghsm)} \notag\\
	&\leq \| I_h \Nsm (\hat X_{h,*}^{m+1} - \hat X_{h,*}^{m}) \|_{W^{1,\infty}(\Ghsm)} + \| I_h \Tsm (\hat X_{h,*}^{m+1} - \hat X_{h,*}^{m}) \|_{W^{1,\infty}(\Ghsm)} \notag\\
	&\lesssim \tau + h^{-1}(\tau^2 + \| I_h \Tsm (\hat X_{h,*}^{m+1} - \hat X_{h,*}^{m}) \|_{L^{\infty}(\Ghsm)}^2) + \| I_h \Tsm (\hat X_{h,*}^{m+1} - \hat X_{h,*}^{m}) \|_{W^{1,\infty}(\Ghsm)} \notag\\
	&\hspace{147pt}\mbox{(by the inverse inequality and \eqref{eq:geo_rel_4}--\eqref{eq:geo_rel_5})} \notag\\
	&\leq (1+h^{-1}\tau) \tau 
	+
	h^{-1}	\| I_h \Tsm (\hat X_{h,*}^{m+1} - \hat X_{h,*}^{m}) \|_{L^{\infty}(\Ghsm)}^2
	\notag\\
	&\quad
	+ \| I_h \Tsm ( X_{h}^{m+1} - X_{h}^{m}) \|_{W^{1,\infty}(\Ghsm)} +  \| I_h \Tsm ((\NsM\circ\hat X_{h,*}^{m+1} - \Nsm\circ\hat X_{h,*}^{m})\ehM) \|_{W^{1,\infty}(\Ghsm)} \notag\\
	&\hspace{186pt}\mbox{(using \eqref{eq:geo_rel_6})} \notag\\
	&\lesssim 
	\tau 
	+ 
	\tau \| I_h \Tsm \delta_\tau\ehm \|_{W^{1,\infty}(\Ghsm)} 
	\quad\,\mbox{(using \eqref{W1infty-g} and \eqref{eq:geo_rel_3})} 
	\notag\\
	&\quad
	+
	h^{-1}	\| I_h \Tsm (\hat X_{h,*}^{m+1} - \hat X_{h,*}^{m}) \|_{L^{\infty}(\Ghsm)}^2
	\notag\\
	&\qquad+ h^{-1/2}\| \hat X_{h,*}^{m+1} - \hat X_{h,*}^{m} \|_{W^{1,\infty}(\Ghsm)} \| \ehM \|_{H^{1}(\Ghsm)}
	,
\end{align}
	where the last inequality results from the following derivation using the super-approximation (cf. Lemma~\ref{lemma:super_conv-nonlinear}):
	\begin{align}
		&\| I_h [\Tsm (\NsM\circ\hat X_{h,*}^{m+1} - \Nsm\circ\hat X_{h,*}^{m})\ehM] \|_{W^{1,\infty}(\Ghsm)}
		\notag\\
		&\lesssim
		\| (1-I_h) [I_h\Tsm (\NsM\circ\hat X_{h,*}^{m+1} - \Nsm\circ\hat X_{h,*}^{m})\ehM] \|_{W^{1,\infty}(\Ghsm)}
		\notag\\
		&\quad+
		\| I_h \Tsm (\NsM\circ\hat X_{h,*}^{m+1} - \Nsm\circ\hat X_{h,*}^{m})\ehM \|_{W^{1,\infty}(\Ghsm)}
		\notag\\
		&\lesssim
		\| I_h\Tsm (\NsM\circ\hat X_{h,*}^{m+1} - \Nsm\circ\hat X_{h,*}^{m}) \|_{W^{1,\infty}(\Ghsm)}
		\| \ehM \|_{L^{\infty}(\Ghsm)}
		\notag\\
		&\quad+
		\| I_h\Tsm (\NsM\circ\hat X_{h,*}^{m+1} - \Nsm\circ\hat X_{h,*}^{m}) \|_{L^{\infty}(\Ghsm)}
		\| \ehM \|_{W^{1,\infty}(\Ghsm)}
		\notag\\
		&\lesssim
		(\tau+\| \hat X_{h,*}^{m+1} -  \hat X_{h,*}^{m} \|_{W^{1,\infty}(\Ghsm)})
		\| \ehM \|_{L^{\infty}(\Ghsm)}
		\notag\\
		&\quad+
		(\tau+\| \hat X_{h,*}^{m+1} -  \hat X_{h,*}^{m} \|_{L^{\infty}(\Ghsm)})
		\| \ehM \|_{W^{1,\infty}(\Ghsm)}. \notag
	\end{align}
Using the inverse inequality, \eqref{eq:ehM-L2-prior} and \eqref{eq:hat_X_s_diff}, the last two terms on the right-hand side of \eqref{eq:hat-X-diff-W1inf} can be absorbed into its left-hand side, and we have
\begin{align}\label{eq:hat-X-diff-W1inf1}
	\| \hat X_{h,*}^{m+1} - \hat X_{h,*}^{m} \|_{W^{1,\infty}(\Ghsm)}
	\lesssim \
	\tau + \tau \| I_h \Tsm \delta_\tau\ehm \|_{W^{1,\infty}(\Ghsm)}
	,
\end{align}
and the same argument gives the $H^1$ bound:
\begin{align}
	\| \hat X_{h,*}^{m+1} - \hat X_{h,*}^{m} \|_{H^{1}(\Ghsm)} 
	\lesssim \tau + \tau \| \nabla_\Ghsm I_h \Tsm \delta_\tau\ehm \|_{L^{2}(\Ghsm)} 
	. \notag
\end{align}
Substitution of \eqref{eq:a-priori-delta-e} yields the smallness estimates
\begin{align}
	\| \hat X_{h,*}^{m+1} - \hat X_{h,*}^{m} \|_{W^{1,\infty}(\Ghsm)} &\lesssim h^{1/2} , \notag
\end{align} 
and
\begin{align}
	\| X_{h}^{m+1} - X_{h}^{m} \|_{W^{1,\infty}(\Ghsm)} 
	&
	\leq
	\tau \| I_h (H^m n^m) \|_{W^{1,\infty}(\Ghsm)}
	+
	\tau \| I_h g^m \|_{W^{1,\infty}(\Ghsm)}
	+
	 \tau\| \delta_\tau \ehm \|_{W^{1,\infty}(\Ghsm)} 
	\lesssim h^{1/4} . \notag
\end{align}
Moreover, from the induction hypothesis \eqref{eq:ind_hypo1},
\begin{align}
	\|X_h^m - \hat X_{h,*}^m \|_{W^{1,\infty}(\Ghsm)}=\|\ehm \|_{W^{1,\infty}(\Ghsm)}\lesssim h^{1/4} . \notag
\end{align}
The preceding smallness estimates and Lemma~\ref{lemma:norm-equiv} show that, for all $p\in[1,\infty]$, the $L^p$ and $W^{1,p}$ norms of finite element functions with a common nodal vector are equivalent on $\Ghm, \GhM, \Ghsm, \GhsM$ and $\Gamma_{h,*}^{m+1}$. 


\subsection{Main error estimates}

In this section, we combine the consistency and stability results to conclude the main error estimate up to a shape regularity constant.
The norm conversion lemma below is proved in Appendix \ref{sec:e-convert}.
\begin{lemma}\label{lemma:e-convert}
	For any $m\geq0$, we have
	\begin{align}
		&\| \ehm \|_{L^2_h(\hat\Gamma_{h,*}^{m})}^2 - \| e_h^m \|_{L^2_h(\hat\Gamma_{h,*}^{m-1})}^2
		\notag\\
		&
		\lesssim
		\tau \Big( 
		\epsilon^{-1}\| \ehm \|_{L^2(\Ghsm)}^2 
		+ 
		\epsilon^{-1} \|  e_h^{m} \|_{L^2(\Ghsm)}^2
		+ 
		\epsilon^{-1} \| \hat e_h^{m-1} \|_{L^2(\Ghsm)}^2
		+ 
		\epsilon^{-1} \|  e_h^{m-1} \|_{L^2(\Ghsm)}^2
		\notag\\
		&\qquad
		+
		\epsilon\| \nabla_\Ghsm \ehm \|_{L^2(\Ghsm)}^2 
		+ 
		\epsilon \| \nabla_\Ghsm e_h^{m} \|_{L^2(\Ghsm)}^2
		+ 
		\epsilon \| \nabla_\Ghsm \hat e_h^{m-1} \|_{L^2(\Ghsm)}^2
		+ 
		\epsilon \| \nabla_\Ghsm e_h^{m-1} \|_{L^2(\Ghsm)}^2
		\Big) ,
		\notag
	\end{align}
	where, for $m=0$, we use the convention $\hat\Gamma_{h,*}^{-1} = \hat\Gamma_{h,*}^{0}$, $e_h^{-1} = 0$, $\hat e_h^{-1} = 0$ and $e_h^0 = \hat e_h^0$.
\end{lemma}
We first observe that
\begin{align}
	& \frac{1}{\tau} (\| \eM \|_{L_h^2(\hat\Gamma_{h,*}^{m})}^2 - \| e_h^m \|_{L^2_h(\hat\Gamma_{h,*}^{m-1})}^2) + \mathscr A_{h, *}(e_h^{m+1},\eM) \notag\\
	&= \frac{1}{\tau}(\| \eM \|_{L_h^2(\hat\Gamma_{h,*}^{m})}^2 - \| \ehm \|_{L^2_h(\hat\Gamma_{h,*}^{m})}^2)  + \mathscr A_{h, *}(e_h^{m+1},\eM) \notag\\
	&\quad+ \frac{1}{\tau}(\| \ehm \|_{L^2_h(\hat\Gamma_{h,*}^{m})}^2
	-
	\| e_h^m \|_{L^2_h(\hat\Gamma_{h,*}^{m-1})}^2
	) . \notag
\end{align}
Then, by testing the error equation \eqref{eq:err_eq2} with $\phi_h = \eM$ and applying the linear and bilinear estimates developed in Sections~\ref{sec:cons_err} and \ref{sec:lin-bilin-est} and the norm conversion lemma (Lemma~\ref{lemma:e-convert}), we continue the derivation as follows:
\begin{align}
	& \frac{1}{\tau}(\| \eM \|_{L_h^2(\hat\Gamma_{h,*}^{m})}^2 - \| e_h^m \|_{L^2_h(\hat\Gamma_{h,*}^{m-1})}^2) + \mathscr A_{h, *}(e_h^{m+1},\eM) \notag\\
	&\lesssim \mathscr A_{h, *}^T(\hat{e}_h^m,\ehm) - \mathscr B^m(\hat{e}_h^m,\eM) - \mathscr J^m(\eM) - \mathscr K^m(\eM) - \mathscr D^m(\eM) \notag\\
	&\quad- \sum_{i=1}^3 \mathscr F_i^m(\eM) + \mathscr A_{h, *}^N(\hat e_h^{m}, I_h \bar T_{h,*}^m \eM) + \mathscr B^m(\hat e_h^{m}, I_h \bar T_{h,*}^m \eM) + \mathscr Q^m(I_h \bar T_{h,*}^m \eM) \notag\\
	&\quad
	+
	\frac{1}{\tau}(\| \ehm \|_{L^2_h(\hat\Gamma_{h,*}^{m})}^2
	-
	\| e_h^m \|_{L^2_h(\hat\Gamma_{h,*}^{m-1})}^2)
	\notag\\
	&\lesssim \epsilon^{-1} (\tau + h^{2})^2 + \epsilon^{-1}\big(\| \hat e_h^m \|_{L^2(\Ghsm)}^2 + \| e_h^m \|_{L^2(\Ghsm)}^2 + \| \hat e_h^{m-1} \|_{L^2(\hat\Gamma_{h, *}^{m})}^2 + \| e_h^{m-1} \|_{L^2(\Ghsm)}^2  \big) \notag\\
	&\quad+  \epsilon \big(\| \nabla_{\Ghsm} \hat e_h^{m} \|_{L^2(\Ghsm)}^2 + \| \nabla_{\Ghsm} e_h^{m} \|_{L^2(\Ghsm)}^2 + \| \nabla_{\Ghsm} \hat e_h^{m-1} \|_{L^2(\Ghsm)}^2 + \| \nabla_{\Ghsm} e_h^{m-1} \|_{L^2(\Ghsm)}^2 \big) . \notag
\end{align} 
Using the error conversion bounds \eqref{eq:ehM-L2-prior}--\eqref{eq:ehM-H1-prior}, the above estimate reduces to
\begin{align}
	& \frac{ \| \eM \|_{L_h^2(\hat\Gamma_{h,*}^{m})}^2 - \| e_h^m \|_{L^2_h(\hat\Gamma_{h,*}^{m-1})}^2 }{\tau}
	+ C^{-1} \| \nabla_{\Ghsm} \eM \|_{L^2(\Ghsm)}^2  \notag\\
	&\lesssim
	\epsilon^{-1} (\tau+ h^{2})^2
		+ \epsilon^{-1}\| e_h^m \|_{L_h^2(\Ghsm)}^2 + \epsilon^{-1}\| e_h^{m-1} \|_{L_h^2(\Ghsm)}^2
		\notag\\
		&\quad
		 + \epsilon\| \nabla_{\Ghsm} e_h^m \|_{L^2(\Ghsm)}^2 + \epsilon\| \nabla_{\Ghsm} e_h^{m-1} \|_{L^2(\Ghsm)}^2 , \notag
\end{align} 
where $\epsilon$ is an arbitrarily small positive constant. Finally, the discrete Gr\"onwall inequality, the norm equivalence from Section~\ref{sec:disc-est} and the initial approximation properties yield
\begin{align}\label{eq:err_fin1}
	\max_{0\le m\le l+1} \| e_h^{m}  \|_{L^2(\Ghsm)}^2 
	+  \sum\limits_{m = 0}^{l+1} \tau \| \nabla_{\Ghsm} 
	e_h^{m} \|_{L^2(\Ghsm)}^2 
	\le C_{\kappa_l}  ( \tau +  h^{2})^2 ,
\end{align}
for sufficiently small $h\le h_{\kappa_l}$. Using the norm conversion formulas \eqref{eq:hat-stab} and \eqref{eq:ehM-H1-prior}, we can convert $e_h^m$ to $\ehm$ in \eqref{eq:err_fin1}. It remains to prove that the shape-regularity constant $\kappa_l$ is bounded uniformly in time; this is the purpose of the next section.


\subsection{Shape regularity analysis}
\label{sec:bbd}



According to the geometric relations \eqref{eq:geo_rel_4}--\eqref{eq:geo_rel_6}, we decompose as follows:
\begin{align*}
&\|\hat X_{h,*}^{m + 1} -\hat X_{h,*}^{m}\|_{W^{1,\infty}(\Gamma_{h,\rm f}^0)} \notag\\
&\le \| I_h [(Y^{m + 1} - {\rm id}_\Gm)\circ a^m \circ \hat X_{h,*}^{m} ]\|_{W^{1,\infty}(\Gamma_{h,\rm f}^0)}
+ \| \rho_h^m\circ \hat X_{h,*}^{m} \|_{W^{1,\infty}(\Gamma_{h,\rm f}^0)} \\
&\quad+\| I_h[ I_h(T_*^m\circ \hat X_{h,*}^{m}) I_h (N_*^{m+1}\circ \hat X_{h,*}^{m+1}-N_*^{m}\circ \hat X_{h,*}^{m})  (\hat e_{h}^{m + 1} \circ \hat X_{h,*}^{m})]\|_{W^{1,\infty}(\Gamma_{h,\rm f}^0)} \\
&\quad+\| I_h[(T_*^m \circ \hat X_{h,*}^{m} ) (X_{h}^{m + 1} - X_{h}^{m})  ] \|_{W^{1,\infty}(\Gamma_{h,\rm f}^0)} \\
&=: \mathscr E_1^m + \mathscr E_2^m + \mathscr E_3^m. 
\end{align*} 
The stability of $I_h$ on $C^0(\Gamma_{h,\rm f}^0)\cap W^{1,\infty}(\Gamma_{h,\rm f}^0)$, the chain rule, the inverse inequality, and \eqref{eq:geo_rel_5} yield
\begin{align*}
\mathscr E_1^m 
&\le C_0 \| (Y^{m + 1} - {\rm id}_\Gm)\circ a^m \circ \hat X_{h,*}^{m} \|_{W^{1,\infty}(\Gamma_{h,\rm f}^0)}
+
 \| \rho_h^m\circ \hat X_{h,*}^{m} \|_{W^{1,\infty}(\Gamma_{h,\rm f}^0)}
 \\
&\le C_0 \| \nabla_{\Gm} [(Y^{m + 1} - {\rm id}_\Gm)]\circ a^m \circ \hat X_{h,*}^{m} 
\,\nabla_{\Gamma_{h,\rm f}^0} (a^m \circ\hat X_{h,*}^{m}) \|_{L^{\infty}(\Gamma_{h,\rm f}^0)} \\
&\quad
+ C_0 \| Y^{m + 1} - {\rm id}_\Gm \|_{L^{\infty}(\Gm)}  
+ C_0 h^{-1} \| \rho_h^m\circ \hat X_{h,*}^{m} \|_{L^{\infty}(\Gamma_{h,\rm f}^0)}
\\
&\le C_0\tau\|\hat X_{h,*}^{m} \|_{W^{1,\infty}(\Gamma_{h,\rm f}^0)} 
+ C_0\tau + C_0 h^{-1} \big(\tau^2 + \|I_h \Tsm (\hat X_{h,*}^{m+1} - \hat X_{h,*}^{m}) \|_{L^\infty(\Ghsm)}^2\big) \notag\\
&\le C_0\tau\|\hat X_{h,*}^{m} \|_{W^{1,\infty}(\Gamma_{h,\rm f}^0)} 
+ C_0(1+h^{-1}\tau)\tau +  \frac{1}{4}\|I_h \Tsm (\hat X_{h,*}^{m+1} - \hat X_{h,*}^{m}) \|_{L^\infty(\Ghsm)}
,
\end{align*}
where we have used \eqref{eq:hat_X_s_diff} in the last inequality.

The $E_2^m$ term can be estimated as
\begin{align*}
\mathscr E_2^m 
&\leq
(\tau+\|I_h \Tsm (\hat X_{h,*}^{m+1} - \hat X_{h,*}^{m}) \|_{L^\infty(\Ghso)}) \| \nabla_\Ghso \ehM \|_{L^{\infty}(\Ghso)}
\notag\\
&\quad
+
(\tau+\|I_h \Tsm (\hat X_{h,*}^{m+1} - \hat X_{h,*}^{m}) \|_{W^{1,\infty}(\Ghso)}) \| \ehM \|_{L^{\infty}(\Ghso)}
\notag\\
&\leq
\tau \| \ehM \|_{W^{1,\infty}(\Ghso)}
+\frac{1}{4}\|I_h \Tsm (\hat X_{h,*}^{m+1} - \hat X_{h,*}^{m}) \|_{W^{1,\infty}(\Ghso)} .
\end{align*} 
Here we have used the nonlinear super-approximation estimate (cf. Lemma~\ref{lemma:super_conv-nonlinear}) for $N_*^{m+1}\circ \hat X_{h,*}^{m+1}-N_*^{m}\circ \hat X_{h,*}^{m}$. We also have used the boundedness $\| \hat e_{h}^{m + 1} \circ \hat X_{h,*}^{m} \|_{W^{1,\infty}(\Gamma_{h,\rm f}^0)}\le C_{\kappa_l} h^{-3/2}(\tau+h^{2})$, which follows from the main error estimate \eqref{eq:err_fin1} and the inverse inequality.

Lastly, for $\mathscr E_3^m$, we define $g^m_{\rm f} = g^m \circ a^m \circ \hat X_{h,*}^m$ on $\Gamma_{h,\rm f}^0$ and derive
\begin{align*} 
\mathscr E_3^m&=\| I_h[(T_*^m \circ \hat X_{h,*}^{m} ) (X_{h}^{m + 1} - X_{h}^{m})  ]\|_{W^{1, \infty}(\Gamma_{h,\rm f}^0)} \notag\\
&\leq \| I_h[(T_*^m \circ \hat X_{h,*}^{m} ) (X_{h}^{m+1} - X_{h}^{m} + \tau I_h[ (H^m n^m)\circ a^m\circ\hat X_{h,*}^{m}]) ] \|_{W^{1, \infty}(\Gamma_{h,\rm f}^0)} 
\notag\\
&\quad
	+ \tau\| I_h[(T_*^m \circ \hat X_{h,*}^{m} )((H^m n^m)\circ a^m\circ\hat X_{h,*}^{m})] \|_{W^{1, \infty}(\Gamma_{h,\rm f}^0)} \notag\\
&= \| I_h[(T_*^m \circ \hat X_{h,*}^{m} ) (\eM - \ehm  + \tau I_h g^m_{\rm f} )]  \|_{W^{1, \infty}(\Gamma_{h,\rm f}^0)} 
\\
&\leq 
C_{\kappa_l}\tau^2
+
\| I_h[(T_*^m \circ \hat X_{h,*}^{m} ) (\eM - \ehm )]  \|_{W^{1, \infty}(\Gamma_{h,\rm f}^0)}  ,
\end{align*}  
where the last inequality follows from
\begin{align*}
\| I_h g^m_{\rm f} \|_{W^{1, \infty}(\Gamma_{h,\rm f}^0)}  &\le C_{\kappa_l} \tau 
,\\
(\Tsm\circ \hat X_{h,*}^m) ((H^m n^m)\circ a^m\circ \hat X_{h,*}^m) &=0
\qquad\mbox{on $\Ghso$} 
. 
\end{align*}

Collecting the estimates for the $E^m$ terms and using absorption gives
\begin{align}\label{eq:X-hat-diff-W1inf}
	&
	\|\hat X_{h,*}^{m + 1} -\hat X_{h,*}^{m}\|_{W^{1,\infty}(\Gamma_{h,\rm f}^0)} 
	\leq  
	C_0 \tau (1 + \| \hat X_{h,*}^m \|_{W^{1, \infty}(\Ghso)}) 
	+
	C_0\tau\| I_h\Tsm\delta_\tau\ehm \|_{W^{1,\infty}(\Gamma_{h,\rm f}^0)}  ,
\end{align}
for sufficiently small $h\leq h_{\kappa_l}$.

Using the triangle inequality, \eqref{eq:X-hat-diff-W1inf}, the main error estimate \eqref{eq:err_fin1} and the step-size condition $\tau\simeq h^2$, we derive that for sufficiently small $h\leq h_{\kappa_l}$
\begin{align}\label{eq:X-hat-diff-triangle}
	&\|  \hat X_{h,*}^{m+1} \|_{W^{1, \infty}(\Ghso)} - \|  \hat X_{h,*}^0 \|_{W^{1, \infty}(\Ghso)} \notag\\
	&\le \sum_{j=0}^m\|\hat X_{h,*}^{j + 1} -\hat X_{h,*}^{j}\|_{W^{1,\infty}(\Gamma_{h,\rm f}^0)} \notag\\ 
	&\leq
	C_0 + \sum_{j=0}^m C_0 \tau \| \hat X_{h,*}^j \|_{W^{1, \infty}(\Gamma_{h,\rm f}^0)}
	+
	\sum_{j=0}^m C_0
	\tau\| I_hT_*^j\delta_\tau\hat e_h^j \|_{W^{1,\infty}(\Gamma_{h,\rm f}^0)} 
	\notag\\
	&\le
	C_0 + \sum_{j=0}^m C_0 \tau \| \hat X_{h,*}^j \|_{W^{1, \infty}(\Gamma_{h,\rm f}^0)}
	+ C_{\kappa_l} h^{1/2} \sum_{j=0}^m \tau \notag\\
	&\quad+ 
	C_{\kappa_l} h^{-3/2} \sum_{j=0}^m \tau \| \nabla_{\hat\Gamma_{h,*}^j} \hat e_h^j \|_{L^2(\hat\Gamma_{h,*}^j)}
	+
	C_{\kappa_l} h^{-3} \sum_{j=0}^m \tau \| \nabla_{\hat\Gamma_{h,*}^j} \hat e_h^j \|_{L^2(\hat\Gamma_{h,*}^j)}^2 \notag\\
	&\le
	C_0 + \sum_{j=0}^m C_0 \tau \| \hat X_{h,*}^j \|_{W^{1, \infty}(\Gamma_{h,\rm f}^0)} 
	\notag\\
	& \quad+
	C_{\kappa_l}
	h^{1/2}
	+
	C_{\kappa_l} h^{-3/2} (\tau+h^2) 
	+
	C_{\kappa_l} h^{-3} (\tau+h^{2})^2 \quad\mbox{(using \eqref{eq:err_fin1})} \notag\\
	&\le 
	C_0 + \sum_{j=0}^m C_0 \tau \| \hat X_{h,*}^j \|_{W^{1, \infty}(\Gamma_{h,\rm f}^0)} 
	\quad\mbox{(by step-size condition $\tau=O(h^2)$)} .
\end{align}
Then, the discrete Gr\"onwall's inequality yields the $W^{1,\infty}$ shape-regularity boundedness
\begin{align}\label{eq:W1inf-shape-reg}
	\max_{0\le m\le l} \|  \hat X_{h,*}^{m+1} \|_{W^{1, \infty}(\Ghso)} 
	&\le 
	C_0 ,
\end{align}
for sufficiently small $h\leq h_{\kappa_l}$.

For the inverse $W^{1,\infty}$ bound, first observe that
\begin{align*}
	& \| (Y^{m + 1} - {\rm id}_\Gm)\circ a^m \circ \hat X_{h,*}^{m} \|_{W^{1,\infty}(\Ghsm)}
	\\
	&\le \| \nabla_{\Gm} [(Y^{m + 1} - {\rm id}_\Gm)]\circ a^m \circ \hat X_{h,*}^{m} 
	\,\nabla_{\Ghsm} (a^m \circ\hat X_{h,*}^{m}) \|_{L^{\infty}(\Ghsm)} 
	+  \| Y^{m + 1} - {\rm id}_\Gm \|_{L^{\infty}(\Gm)}  
	\\
	&\le C_0\tau\| {\rm id}_\Ghsm \|_{W^{1,\infty}(\Ghsm)} 
	+ C_0\tau 
	\le C_0\tau
	.
\end{align*}
Repeating the derivation of \eqref{eq:X-hat-diff-W1inf} and using norm equivalence and the bound \eqref{eq:W1inf-shape-reg} gives
\begin{align}\label{eq:X-hat-diff-W1inf1}
		&
		\|\hat X_{h,*}^{m + 1} -\hat X_{h,*}^{m}\|_{W^{1,\infty}(\GhsM)} 
		\leq
		C_0
		\|\hat X_{h,*}^{m + 1} -\hat X_{h,*}^{m}\|_{W^{1,\infty}(\Ghsm)} 
		\notag\\
		&\leq  
		C_0 \tau 
		+
		C_{\kappa_l}\tau\| I_h\Tsm\delta_\tau\ehm \|_{W^{1,\infty}(\Gamma_{h,\rm f}^0)} 
		,
\end{align}
for sufficiently small $h\leq h_{\kappa_l}$.
The chain rule and \eqref{eq:X-hat-diff-W1inf1} then give
\begin{align}
	&\| (\hat X_{h,*}^{m+1})^{-1} \|_{W^{1,\infty}(\hat\Gamma_{h,*}^{m+1})}
	=
	\| (\hat X_{h,*}^{m})^{-1} \circ
	({\rm id}_\GhsM - (\hat X_{h,*}^{m+1} - \hat X_{h,*}^{m})) \|_{W^{1,\infty}(\hat\Gamma_{h,*}^{m+1})}
	\notag\\
	&\leq
	\| (\hat X_{h,*}^{m})^{-1} \|_{W^{1,\infty}(\hat\Gamma_{h,*}^{m})} \|
	{\rm id}_\GhsM - (\hat X_{h,*}^{m+1} - \hat X_{h,*}^{m}) \|_{W^{1,\infty}(\hat\Gamma_{h,*}^{m+1})}
	\notag\\
	&\leq
	(1+\|
	\hat X_{h,*}^{m+1} - \hat X_{h,*}^{m} \|_{W^{1,\infty}(\hat\Gamma_{h,*}^{m+1})})
	\| (\hat X_{h,*}^{m})^{-1} \|_{W^{1,\infty}(\hat\Gamma_{h,*}^{m})} 
	\notag\\
	&\leq
	(1+C_0\tau)
	\| (\hat X_{h,*}^{m})^{-1} \|_{W^{1,\infty}(\hat\Gamma_{h,*}^{m})} 
	+
	C_{\kappa_l}\tau\| I_h\Tsm\delta_\tau\ehm \|_{W^{1,\infty}(\Gamma_{h,\rm f}^0)}  .
\end{align}
Arguing as in the derivation of \eqref{eq:X-hat-diff-triangle} and applying the discrete Gronwall inequality, we obtain the inverse $W^{1,\infty}$ bound
\begin{align}\label{eq:W1inf-inv-shape-reg}
	\max_{0\le m\le l} \|  (\hat X_{h,*}^{m+1})^{-1} \|_{W^{1, \infty}(\GhsM)} 
	&\le 
	C_0 ,
\end{align}
for sufficiently small $h\leq h_{\kappa_l}$.

Together, \eqref{eq:W1inf-shape-reg} and \eqref{eq:W1inf-inv-shape-reg} give
\begin{align}
	\kappa_{l+1}
	&\le 
	C_0 , \notag
\end{align}
and hence the mesh-size requirement can be improved to $h\leq h_{0}$, where $h_0$ is independent of $\tau$ and $l$.
For sufficiently small $h \leq h_{0}$, the induction hypothesis \eqref{eq:ind_hypo1} at time level $l+1$ is recovered.
The proof of Theorem \ref{thm:main} is now complete.
\hfill\endproof

\appendix
\section{Notation}
\label{section:notation}
The notation below is frequently used in this article.
\begin{longtable}{p{1.1cm}p{13cm}}
	$\Gamma^m$:
	& 
	The exact smooth curve at time level $t=t_m$.\\
	
	$\Gamma_h^m$:
	&
	The numerically computed curve at time level $t=t_m$.\\
	
	$\bfx^{m}$: 
	&
	The nodal vector $\bfx^m=(x_1^m,\dots,x_J^m)^\top$ consisting of the positions of nodes on $\Gamma_h^m$.\\
	
	$\hat\bfx_*^{m}$: 
	&
	The distance projection of $\bfx^{m}$ onto the exact curve $\Gamma^m$, i.e., $\hat\bfx_*^{m}=(\hat x_{1,*}^m,\dots,\hat x_{J,*}^m)^\top$ with $\hat x_{j,*}^m=a^m(x_j^m)$. \\
	
	$\bfx_*^{m+1}$: 
	&
	The new position of $\hat\bfx_*^{m}$ evolving under curve-shortening flow (without additional tangential motion) from $t_m$ to $t_{m+1}$. \\
	
	$\Ghsm$: 
	&
	The piecewise polynomial curve which interpolates $\Gamma^m$ at the nodes in $\hat\bfx_*^{m}$.\\
	
	$\Gamma_{h,*}^{m+1}$: 
	&
	The piecewise polynomial curve which interpolates $\Gamma^{m+1}$ at the nodes in $\bfx_*^{m+1}$.\\
	
	$X_{h}^{m}$: 
	&
	The finite element function with nodal vector $\bfx^m$. It coincides with the identity map, i.e., ${\rm id}(x)=x$, when it is considered as a function on $\Gamma_{h}^m$. \\
	
	$X_{h}^{m+1}$: 
	&
	The finite element function with nodal vector $\bfx^{m+1}$. 
	When it is considered as a function on $\Gamma_{h}^m$, it represents the local flow map from $\Gamma_{h}^m$ to $\Gamma_{h}^{m+1}$.\\
	
	$\hat X_{h,*}^{m}$: 
	&
	The finite element function with nodal vector $\hat\bfx_*^m$. It coincides with the identity map, i.e., ${\rm id}(x)=x$, when it is considered as a function on $\Ghsm$. It coincides with the discrete flow map from $\hat\Gamma_{h,*}^0$ to $\Ghsm$ when it is considered as a function on $\hat\Gamma_{h,*}^0$.\\
	
	$X_{h,*}^{m+1}$: 
	&
	The finite element function with nodal vector $\bfx_*^{m+1}$. 
	When it is considered as a function on $\Ghsm$, it represents the local flow map from $\Ghsm$ to $\Gamma_{h,*}^{m+1}$.\\
	
	$X^{m+1}$: 
	&
	The global flow map from $\Gamma^0$ to $\Gamma^{m+1}$ under curve-shortening flow. \\
	
	$Y^{m+1}$: 
	&
	The local flow map from $\Gamma^m$ to $\Gamma^{m+1}$ under curve-shortening flow. \\
	
	$\hat e_{h}^m$: 
	&
	The finite element error function with nodal vector $\hat\bfe^m=\bfx^m-\hat\bfx_*^m$.\\ 
	
	$e_{h}^{m+1}$: 
	&
	The auxiliary error function with nodal vector $\bfe^{m+1}=\bfx^{m+1}-\bfx_*^{m+1}$.\\ 
	
	$n^m$: 
	&
	The unit normal vector on $\Gamma^m$. \\
	
	$n^m_*$: 
	&
	The unit normal vector of $\Gamma^m$ inversely lifted to a neighborhood of $\Gamma^m$ (including $\Ghsm$), i.e., $n^m_*=n^m\circ a^m$. \\
	
	$\hat n_{h,*}^m$: 
	&
	The normal vector on $\Ghsm$. \\
	
	$\bar n_{h,*}^m$: 
	&
	The averaged normal vector on $\Ghsm$, which is not necessarily unit. \\
	
	$n_h^m$: 
	&
	The normal vector on $\Gamma_{h}^m$. \\
	
	$\bar n_h^m$: 
	&
	The averaged normal vector on $\Gamma_{h}^m$, which is not necessarily unit. \\
	
	$\hat \mu_{h,*}^m$: 
	&
	The co-normal vector (unit tangent vector) on $\Ghsm$. \\
	
	$\mu_{h}^m$: 
	&
	The co-normal vector (unit tangent vector) on $\Ghm$. \\
	
	$\Nsm$: 
	&
	The normal projection operator $\Nsm=n^m_* (n^m_*)^\top$ on $\Ghsm$. \\
	
	$N^m$: 
	&
	The normal projection operator $N^m=n^m (n^m)^\top$ on $\Gamma^m$. 
	Thus $N^m$ is the lift of $\Nsm$ onto $\Gamma^m$, and $N_*^m$ is the extension of $N^m$ to a neighborhood of $\Gm$. \\
	
	$\Nhsm$: 
	&
	The normal projection operator $\Nhsm=\hat n^m_{h,*} (\hat n^m_{h,*})^\top$ on $\Ghsm$. \\
	
	$\Nbhsm$: 
	&
	The averaged normal projection operator $\Nbhsm= \frac{\bar n^m_{h,*}}{|\bar n^m_{h,*}|} (\frac{\bar n^m_{h,*}}{|\bar n^m_{h,*}|})^\top$ on $\Ghsm$. \\
	
	$T_*^m$: 
	&
	The tangential projection operator $T_*^m=I - n^m_* (n^m_*)^\top$ on $\Ghsm$. \\
	
	$T^m$: 
	&
	The tangential projection operator $T^m=I - n^m (n^m)^\top$ on $\Gamma^m$. 
	Thus $T^m$ is the lift of $\Tsm$ onto $\Gamma^m$. \\
	
	$\Thsm$: 
	&
	The tangential projection operator $\Thsm=I - \nhsm (\nhsm)^\top$ on $\Ghsm$. \\
	
	$\Tbhsm$: 
	&
	The averaged tangential projection operator $\Tbhsm=I - \frac{\bar n^m_{h,*}}{|\bar n^m_{h,*}|} (\frac{\bar n^m_{h,*}}{|\bar n^m_{h,*}|})^\top$ on $\Ghsm$. \\
	
	$\mathcal{N}(\Gamma_h^m)$: 
	&
	The collection of nodes of $\Gamma_h^m$. 
%
\end{longtable}

\section{Surface calculus formulas}

Let $\Gamma\subset\mathbb R^2$ be a smooth curve, possibly with
boundary. For $u\in C^\infty(\Gamma)$, we denote by
$\ud_i u$, $i=1,2$, the $i$th Cartesian component of
$\nabla_\Gamma u$. The corresponding product rule, chain rule, integration-by-parts formula, commutator identities, and evolution formulas are collected below. We refer to \cite[Lemma~3.14]{BL2025} and the references therein for their proofs.
\begin{lemma}\label{lemma:ud}
	Let $\Gamma$ and $ \Gamma^\prime$ be two smooth curves that are possibly open, such as smooth pieces of some finite element curves, and let $f, h \in C^\infty(\Gamma)$ and $g\in C^\infty(\Gamma^\prime; \Gamma)$ be given functions. Then the following results hold. 
	\begin{itemize}
		\item[1.]  $\ud_i(fh) = \ud_i f h + f\ud_i h$ on $\Gamma$.
		\item[2.] $\ud_i(f\circ g) = (\ud_j f\circ g)\, \ud_i g_j$ on $\Gamma'$.
		\item[3.] $\int_{\Gamma}f \ud_i h = -\int_{\Gamma}\ud_i f h + \int_{\Gamma}f h H n_i + \int_{\partial\Gamma}f h \mu_i$ where $n, \mu$ are the normal and co-normal (tangential) directions, respectively, and $H:=\ud_i n_i$ (with the Einstein notation) is the mean curvature, i.e. the trace of the second fundamental form.
		\item[4.] $\ud_i \ud_j f = \ud_j \ud_i f + n_i H_{jl} \ud_l f -  n_j H_{il} \ud_l f$, where $H_{ij} := \ud_i n_j = \ud_j n_i$.
		\item[5.] If $\Gamma$ evolves under the velocity field $v$, and $G_T := \bigcup_{t\in [0, T]}\Gamma(t) \times \{t\}$, then 
		$$\md(\ud_i f) = \ud_i (\md f )- (\ud_i v_j - n_i n_l \ud_j v_l)\ud_j f \quad\forall\, f \in C^2(G_T) ,$$
		where $\md$ denotes the material derivative with respect to $v$.
		\item[6.] If $f, h \in C^2(G_T)$ then 
		$$\frac{\d}{\d t}\int_{\Gamma} f h = \int_{\Gamma} \md f h + \int_{\Gamma} f\md h + \int_{\Gamma} f h (\nabla_\Gamma\cdot v).$$
		The divergence is defined as $\nabla_\Gamma\cdot v := \ud_i v_i$, which coincides with the intrinsic divergence on the curve if $v$ is a tangential vector field on $\Gamma$. Since the Lagrange interpolation commutes with the material time derivative, it is straightforward to check in the local coordinates that an analogous result also holds for the mass lumping integral, i.e., 
		$$
		\frac{\d}{\d t}\int_{\Gamma_h}^h \tilde f \tilde h = \int_{\Gamma_h}^h \md \tilde f \tilde h + \int_{\Gamma_h}^h \tilde f\md \tilde h + \int_{\Gamma_h}^h \tilde f \tilde h (\nabla_{\Gamma_h}\cdot v_h) ,
		$$
		where $\Gamma_h$ is a finite element curve moving with polynomial velocity $v_h\in [S_h(\Gamma_h)]^2$ (mass lumping is well defined on $\Gamma_h$), and $\tilde f,\tilde h$ are continuous functions defined on $\bigcup_{t\in [0, T]}\Gamma_h(t)\times\{t\}$.  
		\item[7.] The evolution of the unit normal vector $n$ of the curve $\Gamma$ with respect to the velocity field $v$ satisfies the following relation: 
		$$
		\md n_i = -\ud_i v_j n_j .
		$$
	\end{itemize}
\end{lemma}

\section{Super-approximation estimates}
\label{sec:super}

This section collects the super-approximation estimates associated with
mass lumping and the projected-error framework; see
\cite{BL2024,BL2025,BGV2027}. Throughout this section,
\[
m\in\big\{0,\ldots,\lfloor T/\tau\rfloor\big\},
\qquad
p,q,r\in[1,\infty].
\]
\begin{lemma}\label{lemma:super_conv}
	The following estimates hold for any piecewise smooth function $f$ and finite element functions $\phi_h,v_h, w_h\in S_h(\Ghsm)$: 
	\begin{align*}
		\| (1 - I_h)(f \phi_h) \|_{L^p(\Ghsm)} &\lesssim \| f \|_{W_h^{2,\infty}(\Ghsm)} h \| \phi_h \|_{L^{p}(\Ghsm)} , \notag\\
		\| \nabla_\Ghsm (1 - I_h)(f \phi_h) \|_{L^p(\Ghsm)} &\lesssim \| f \|_{W_h^{2,\infty}(\Ghsm)} h \| \phi_h \|_{W^{1,p}(\Ghsm)} ,\\
		\| (1 - I_h)(v_h w_h) \|_{L^p(\Ghsm)} &\lesssim  h^2 \| v_h \|_{W^{1,q}(\Ghsm)} \| w_h \|_{W^{1,r}(\Ghsm)} ,\\
		\| \nabla_{\Ghsm}(1 - I_h)(v_h w_h) \|_{L^p(\Ghsm)} &\lesssim  h \| v_h \|_{W^{1,q}(\Ghsm)} \| w_h \|_{W^{1,r}(\Ghsm)} ,
	\end{align*}
	for any $1/p=1/q+1/r$.
\end{lemma}
As a direct consequence of Lemma~\ref{lemma:super_conv}, we obtain the following estimate.
\begin{lemma}\label{lemma:T<=N2}
	The following estimates hold:
	\begin{align} 
		\| \Tsm \ehm \|_{L^p(\Ghsm)}
		&\lesssim
		h\| \ehm \|_{L^p(\Ghsm)}  , \notag\\
		\| \Tsm \ehm \|_{W^{1,p}(\Ghsm)}
		&\lesssim
		h\| \ehm \|_{W^{1,p}(\Ghsm)} 
		\notag .
	\end{align}
\end{lemma}
A nonlinear counterpart of Lemma~\ref{lemma:super_conv} follows from the chain and product rules in Lemma~\ref{lemma:ud}.
\begin{lemma}\label{lemma:super_conv-nonlinear}
	Given a function $f\in W^{3,\infty}(D)$, defined on some open bulk region $D\subset \R^2$, and any vector-valued finite element functions $\phi_h, \psi_h \in [S_h(\Ghso)]^2$, whose ranges are contained in $D$, we have
	\begin{align*}
		\| (1 - I_h)(f\circ\phi_h - f\circ\psi_h) \|_{L^p(\Ghso)} &\leq C h \| \phi_h - \psi_h \|_{L^{p}(\Ghso)} , 
		\notag\\
		\| \nabla_\Ghso (1 - I_h)(f\circ\phi_h - f\circ\psi_h) \|_{L^p(\Ghso)} &\leq C h \| \phi_h - \psi_h \|_{W^{1,p}(\Ghso)}  . 
	\end{align*}
	Consequently, from the triangle inequality and Lipschitz continuity, it holds that
	\begin{align*}
		\| I_h(f\circ\phi_h - f\circ\psi_h) \|_{L^p(\Ghso)} &\leq C \| \phi_h - \psi_h \|_{L^{p}(\Ghso)} , 
		\notag\\
		\| I_h (f\circ\phi_h - f\circ\psi_h) \|_{W^{1,p}(\Ghso)} &\leq C \| \phi_h - \psi_h \|_{W^{1,p}(\Ghso)}  .
	\end{align*}
	The constants $C$ above depend on $\| f \|_{W^{3,\infty}(D)}$, $\| \phi_h \|_{W^{1,\infty}(\Ghso)}$ and $\| \psi_h \|_{W^{1,\infty}(\Ghso)}$.
\end{lemma}
The approximation properties of the Gauss--Lobatto quadrature rule
give the following estimate.
\begin{lemma}\label{lemma:super_conv2}
	Let $f$ be a function which is smooth on every element $K$ of $\Ghsm$, and assume that the pull-back function $f\circ F_K $ vanishes at all the Gauss--Lobatto points of the flat segment $K_{\rm f}^0$ for every element $K$ of $\Ghsm$. Then the following estimate holds: 
	\begin{align}
		\Big|\int_\Ghsm f \d\xi \Big| 
		\lesssim h^{2} \| f \|_{W^{2,1}_h(\Ghsm)} , \notag
	\end{align}
	where $\|\cdot\|_{W^{2,1}_h(\Ghsm)}$ denotes the piecewise $W^{2,1}$ norm. 
\end{lemma}
The following estimate is a direct consequence of
Lemma~\ref{lemma:super_conv2}.
\begin{lemma}\label{Lemma-GLW}
	For a smooth function $f$ on $\Gm$, the following estimate holds: 
	\begin{align*}
		\Big| \int_{\Gm} \nabla_{\Gm} (f - (I_h f^{-\ell})^\ell) \cdot  \nabla_{\Gm}  \phi_h^\ell \Big|
		\lesssim 
		h^{2} \|f\|_{H^{2}(\Gm)} \|\phi_h\|_{H^1(\Ghsm)} 
		\quad\forall\,\phi_h\in S_h(\Ghsm) . 
	\end{align*}
\end{lemma}
Lemma~\ref{lemma:T<=N2} also yields an important cancellation estimate for the tangential stiffness bilinear form.
\begin{lemma}\label{lemma:AT-sup}
	We have
	\begin{align}
		| \mathscr A_{h, *}^T(\hat{e}_h^m,\phi_h) | 
		&\lesssim \min\Big\{\| \ehm \|_{L^2({\Ghsm})}  \| \phi_h \|_{H^1({\Ghsm})} , \| \ehm \|_{H^1({\Ghsm})}  \| \phi_h \|_{L^2({\Ghsm})}\Big\}
		. \notag
	\end{align}
\end{lemma}
\begin{proof}
	Using local integration by parts, we obtain
	\begin{align}
		\mathscr A_{h, *}^T(\hat{e}_h^m,\phi_h)
		&=
		\int_\Ghsm \nabla_\Ghsm \ehm\cdot \Tsm \nabla_\Ghsm \phi_h
		+
		\int_\Ghsm \nabla_\Ghsm \ehm\cdot (\Thsm-\Tsm) \nabla_\Ghsm \phi_h
		\notag\\
		&=
		-
		\int_\Gm  (\nabla_\Gm T^m\cdot (\ehm)^\ell)\cdot \nabla_\Gm (\phi_h)^\ell
		\notag\\
		&\quad
		+
		\int_\Ghsm \nabla_\Ghsm (\ehm\cdot \Tsm)\cdot \nabla_\Ghsm \phi_h
		\notag\\
		&\quad
		-
		\int_\Ghsm  (\nabla_\Ghsm \Tsm\cdot \ehm)\cdot \nabla_\Ghsm \phi_h
		+
		\int_\Gm  (\nabla_\Gm T^m\cdot (\ehm)^\ell)\cdot \nabla_\Gm (\phi_h)^\ell
		\notag\\
		&\quad+
		\int_\Ghsm \nabla_\Ghsm \ehm\cdot (\Thsm-\Tsm) \nabla_\Ghsm \phi_h
		\notag\\
		&=:
		\sum_{j=1}^4 \mathscr A_j(\hat{e}_h^m,\phi_h)
		. \notag
	\end{align}
	Integration by parts on the smooth closed curve $\Gm$ gives
	\begin{align}
		| \mathscr A_{1}(\hat{e}_h^m,\phi_h) | 
		&\lesssim \min\Big\{\| \ehm \|_{L^2({\Ghsm})}  \| \phi_h \|_{H^1({\Ghsm})} , \| \ehm \|_{H^1({\Ghsm})}  \| \phi_h \|_{L^2({\Ghsm})}\Big\}
		. \notag
	\end{align}
	By Lemma~\ref{lemma:T<=N2},
	\begin{align}
		| \mathscr A_{2}(\hat{e}_h^m,\phi_h) | 
		&\lesssim h\| \ehm \|_{H^1({\Ghsm})}  \| \phi_h \|_{H^1({\Ghsm})}
		. \notag
	\end{align}
	Lemma~\ref{lemma:geo-pert} yields
	\begin{align}
		| \mathscr A_{3}(\hat{e}_h^m,\phi_h) | 
		&\lesssim h\| \ehm \|_{L^2({\Ghsm})}  \| \phi_h \|_{H^1({\Ghsm})}
		. \notag
	\end{align}
	Finally, \eqref{normal-intpl} implies
	\begin{align}
		| \mathscr A_{4}(\hat{e}_h^m,\phi_h) | 
		&\lesssim h\| \nabla_\Ghsm \ehm \|_{L^2({\Ghsm})}  \| \nabla_\Ghsm \phi_h \|_{L^2({\Ghsm})}
		. \notag
	\end{align}
	The proof is complete upon collecting all the estimates above.
\end{proof}

\section{Discrete norms}\label{sec:disc-norm}

Since the weights of the Gauss--Lobatto quadrature are positive, the discrete $L^p$ norm defined by 
\begin{align}
	\| v \|_{L^p_h(\Ghsm)}  &:= \Big( \int_\Ghsm^h |v|^p \Big)^{\frac1p} 
	= \Big( \sum\limits_{K \subset \Ghsm} \int_{K_{\rm f}^0} I_{K_{\rm f}^0} \big( |v \circ F_K |^p |\nabla_{K_{\rm f}^0} F_K| \big) \Big)^{\frac1p}\quad p\in[1,\infty),
	\notag\\
	\| v \|_{L^\infty_h(\Ghsm)}  &:= \| v \|_{L^\infty(\Ghsm)}  , \notag
\end{align}
is indeed a norm on the finite element space $S_h(\Ghsm)$ because $\| v \|_{L^p_h(\Ghsm)}=0$ iff $v=0$ at all the nodes of $\Ghsm$. In addition, this discrete $L^p$ norm is also well defined for functions which are piecewise continuous on $\Ghsm$. Its basic properties are summarized below; cf. \cite[Lemma 3.7]{BL2025}. 

\begin{lemma}\label{lemma:lump}
	Given $p\in[1,\infty]$, the following relations hold for all finite element functions $v_h\in S_h(\Ghsm)$ and piecewise continuous functions $w_1, w_2, w_3$ on $\Ghsm$: 
	\begin{align}
		\| v_h \|_{L^p_h(\Ghsm)} &\sim \| v_h \|_{L^p(\Ghsm)} , \notag\\
		\| \nabla_\Ghsm v_h \|_{L^p_h(\Ghsm)} &\sim \| \nabla_\Ghsm v_h \|_{L^p(\Ghsm)} , \notag\\
		\Big|\int_\Ghsm^h w_1 w_2 w_3 \Big| &\lesssim \| w_1 \|_{L^\infty(\Ghsm)} \| w_2 \|_{L^2_h(\Ghsm)} \| w_3 \|_{L^2_h(\Ghsm)}  \notag .
	\end{align}
\end{lemma}


\section{Proof of Lemma~\ref{lemma:e-convert}}
\label{sec:e-convert}

This appendix proves the stability estimate for converting
\[
\|\ehm\|_{L_h^2(\hat\Gamma_{h,*}^m)}^2
\quad\text{into}\quad
\|e_h^m\|_{L_h^2(\hat\Gamma_{h,*}^{m-1})}^2
\]
at each time level. For $m=0$, the conclusion follows directly from
the convention stated in Lemma~\ref{lemma:e-convert}. We therefore
assume that $m\geq1$.

We first decompose the difference as
\begin{align}
	&\| \ehm \|_{L^2_h(\hat\Gamma_{h,*}^{m})}^2 - \| e_h^m \|_{L^2_h(\hat\Gamma_{h,*}^{m-1})}^2 \notag\\
	&= \| e_h^m \|_{L^2_h(\hat\Gamma_{h,*}^{m})}^2 - \| e_h^m \|_{L^2_h(\hat\Gamma_{h,*}^{m-1})}^2 
	\notag\\
	&\quad+ \| \ehm \|_{L^2_h(\hat\Gamma_{h,*}^{m})}^2 - \| e_h^m \|_{L^2_h(\hat\Gamma_{h,*}^{m})}^2 
	\notag\\
	&=: \mathscr M_1^m + \mathscr M_2^m  . \notag
\end{align}
By the fundamental theorem of calculus, the displacement estimate \eqref{eq:hat-X-diff-W1inf1}, the tangential velocity estimate \eqref{eq:tan_stab_H1_e2}, and the norm equivalences established above, we obtain
\begin{align}
	\mathscr M_1^m
	&=  \| e_h^m \|_{L^2_h(\hat\Gamma_{h,*}^{m})}^2 - \| e_h^m \|_{L^2_h(\hat\Gamma_{h,*}^{m-1})}^2   \notag\\
	&\lesssim \| \nabla_{\hat\Gamma_{h,*}^{m-1}} (\hat X_{h,*}^{m} - \hat X_{h,*}^{m-1}) \|_{L^\infty({\hat\Gamma_{h,*}^{m-1}})} \| e_h^m \|_{L^2(\Ghsm)}^2 \notag\\
	&\lesssim 
	\tau(1 + h^{-1/2}\|  \nabla_{\hat\Gamma_{h,*}^{m-1}} I_h T_*^{m-1} \delta_\tau \hat e_h^{m-1} \|_{L^2({\hat\Gamma_{h,*}^{m-1}})}) \| e_h^m \|_{L^2(\Ghsm)}^2 
	\notag\\
	&\lesssim 
	\tau(1 + h^{-3/2}(\tau+h^2) + h^{-3/2} \|  \nabla_{\hat\Gamma_{h,*}^{m-1}} \hat e_h^{m-1} \|_{L^2(\hat\Gamma_{h,*}^{m-1})} + h^{-3} \|  \nabla_{\hat\Gamma_{h,*}^{m-1}} \hat e_h^{m-1} \|_{L^2(\hat\Gamma_{h,*}^{m-1})}^2) \| e_h^m \|_{L^2(\Ghsm)}^2 
	.  \notag
\end{align}
For $\mathscr M_2^m$, we use \eqref{eq:hat-stab} to get
\begin{align}
	\mathscr M_2^m 
	=  \| \ehm \|_{L^2_h(\hat\Gamma_{h,*}^{m})}^2 - \| e_h^m \|_{L^2_h(\hat\Gamma_{h,*}^{m})}^2  
	\leq 0
	. \notag
\end{align}
The proof of Lemma~\ref{lemma:e-convert} is complete by collecting the above estimates and using the norm equivalence, the induction hypothesis \eqref{eq:ind_hypo1}, the step-size condition $\tau\simeq h^2$ and Young's inequality.

\renewcommand{\refname}{\bf References}

\bibliographystyle{abbrv}
\bibliography{MCF}

\end{document}